\documentclass[11pt,a4paper]{article}
\usepackage{amsmath, amsthm}
\usepackage[psamsfonts]{amssymb}
\usepackage[all]{xy}
\usepackage{graphicx}
\usepackage[T1]{fontenc}
\usepackage{hyperref}
\usepackage{verbatim}
\usepackage{color}
\usepackage{tikz}
\xyoption{curve}
\usepackage{enumitem}
\usepackage[a4paper, total={6in, 8in}]{geometry}
\usepackage{mathtools}

\newtheorem{theorem}{Theorem}[section]
\newtheorem{prop}[theorem]{Proposition}
\newtheorem{lemma}[theorem]{Lemma}

\newtheorem{question}[theorem]{Question}

\newtheorem{remark}[theorem]{Remark}

\def\co{\colon\thinspace}

\def\A{\mathcal{A}}
\def\Q{\mathbb{Q}}
\def\R{\mathbb{R}}
\def\Z{\mathbb{Z}}
\def\N{\mathbb{N}}
\def\C{\mathbb{C}}

\def\FR{\operatorname{ind_{FR}}}
\def\EH{\operatorname{ind_{EH}}}

\newcommand{\LL}{\mathcal{L}}
\newcommand{\FF}{\mathcal{F}}
\newcommand{\half}{{\textstyle\frac{1}{2}}}

\numberwithin{equation}{section}

\newtheorem{proposition}[theorem]{Proposition}

\newtheorem{lemma-definition}[theorem]{Lemma-Definition}

\newtheorem{definition}[theorem]{Definition}

\newcommand{\eqdef}{\;{:=}\;}

\newcommand{\mc}[1]{{\mathcal #1}}

\newcommand{\op}{\operatorname}

\newcommand{\M}{\mc{M}}

\newcommand{\End}{\op{End}}

\newcommand{\tensor}{\otimes}

\newcommand{\CZ}{\op{CZ}}

\newcommand{\w}{\omega}

\newcommand{\Per}{\mathcal{P}}
\newcommand{\Mod}{\mathcal{M}}

\newcommand{\Spec}{\operatorname{Spec}}

\newcommand{\Hstd}{\mathcal{H}_{\textrm{std}}}
\newcommand{\Hi}{\mathcal{H}}
\newcommand{\Ac}{\mathcal{A}}
\newcommand{\Id}{\op{Id}}
\newcommand{\codim}{\op{codim}}
\newcommand{\ind}{\op{ind_{EH}}}
\newcommand{\CP}{\mathbb{C}P}

\newcommand{\union}{\bigcup}
\newcommand{\intt}{\bigcap}
\newcommand{\eps}{\varepsilon}

\newcommand{\bpm}{\begin{pmatrix}}
\newcommand{\epm}{\end{pmatrix}}

\renewcommand{\epsilon}{\varepsilon}

\makeatletter
\DeclareRobustCommand\widecheck[1]{{\mathpalette\@widecheck{#1}}}
\def\@widecheck#1#2{%
    \setbox\z@\hbox{\m@th$#1#2$}%
    \setbox\tw@\hbox{\m@th$#1%
       \widehat{%
          \vrule\@width\z@\@height\ht\z@
          \vrule\@height\z@\@width\wd\z@}$}%
    \dp\tw@-\ht\z@
    \@tempdima\ht\z@ \advance\@tempdima2\ht\tw@ \divide\@tempdima\thr@@
    \setbox\tw@\hbox{%
       \raise\@tempdima\hbox{\scalebox{1}[-1]{\lower\@tempdima\box
\tw@}}}%
    {\ooalign{\box\tw@ \cr \box\z@}}}
\makeatother

\makeatletter
\newcommand*{\rom}[1]{\expandafter\@slowromancap\romannumeral #1@}
\makeatother
\title{The equivalence of Ekeland--Hofer and equivariant symplectic homology capacities}
\author{Jean Gutt and Vinicius G. B. Ramos}
\date{}

\begin{document}
\maketitle

\begin{abstract}
In this paper, we prove that the Ekeland--Hofer capacities coincide on all star-shaped domain in $\R^{2n}$ with the equivariant symplectic homology capacities defined by the first author and Hutchings, answering a 35 years old question.
Along the way, we prove that given a Hamiltonian $H$, the (equivariant) Floer homology of $H$ is chain-complex isomorphic to the Morse homology of the Hamiltonian action functional.
\end{abstract}

\section{Introduction}
Inpired by Gromov's seminal work \cite{gromov}, Ekeland and Hofer defined the concept of a symplectic capacity in \cite{EH}. It was the first nontrivial symplectic invariant. If $X,X'$ are domains in $\R^{2n}$, a symplectic embedding from $X$ to $X'$ is a smooth embedding $\varphi:X \to X'$ such that $\varphi^*\omega=\omega$, where $\omega=\sum_{i=1}^n dx_i\wedge dy_i$ is the standard symplectic form on $\R^{2n}$. A \textbf{symplectic capacity} is a function $c$ that assigns to each subset in $\R^{2n}$ a number $c(X)\in[0,\infty]$ satisfying the following axioms:
\begin{description}
	\item{(Monotonicity)} If $X,X' \subset \R^{2n}$, and if there exists a symplectic embedding $X \hookrightarrow X'$, then $c(X) \le c(X')$.
	\item{(Conformality)} If $r$ is a positive real number, then $c(rX) = r^2 c(X)$.	
\end{description}

Associated to a star-shaped domain $X$ in $\R^{2n}$ are two increasing sequences of symplectic capacities: the Ekeland--Hofer capacities \cite{EH2}, $c_k^{EH}(X)$ and the so-called Gutt--Hutchings capacities \cite{GuH}, $c_k^{GH}(X)$.
The Ekeland--Hofer capacities, whose precise definition is recalled in Section \ref{sec:EH}, are defined using calculus of variations whereas the Gutt--Hutchings capacities (see Section \ref{sec:GH}) are defined using some variant of Floer theory.
Inherent to their definition, the Ekeland-Hofer capacities are only defined for domains in $\R^{2n}$ (or quotient of $\R^{2n}$ by discrete groups). The Gutt--Hutchings capacities on the other hand, are defined for a much larger class of symplectic manifolds: Liouville domains.
A question from 1989 which inspired numerous works in geometric analysis is
\begin{question}\label{question:main}
	How can one extend Ekeland--Hofer capacities to more general symplectic manifolds?
\end{question}

\subsection{The main result}
In a talk in 1988, A. Floer mentioned \cite{WeinsteinFloer} that ``the Ekeland--Hofer capacities must come from some type of $S^1$ equivariant homology''.
The main result of this paper is a rigorous statement of Floer's visionary intuition and an answer to Question \ref{question:main}.
\begin{theorem}\label{thmfancystatement}
	Given a star-shaped domain $X$ in $\R^{2n}$, we have for all $k\in\N_0$:
	\[
		c_k^{EH}(X)=c_k^{CH}(X)
	\]
\end{theorem}

Before giving the idea of the proof of Theorem \ref{thmfancystatement}, we shall recall the Definitions of the aforementionned capacities.
\subsection{The Ekeland--Hofer capacities}\label{sec:EH}
Let $S^1=\R/\Z$.
Let $E$ be the vector space defined by
\[E=H^{1/2}(S^1,\R^{2n})=\left\{x=\sum_{k\in\Z} e^{2\pi i k t} x_k\in L^2(S^1,\R^{2n})\mid  \sum_{k\in \Z} |k| |x_k|^2<\infty\right\}.\]
Here the Fourier coefficients are elements of $\R^{2n}=\C^n$.
We recall that $E$ is a Hilbert space with inner product given by
\[\left\langle x,y\right\rangle_{H^{1/2}}=x_0\cdot y_0+2\pi \sum_{k\in\Z} |k|\, x_k\cdot y_k,\]
where $\cdot$ denotes the real scalar product in $\R^{2n}$. In particular, the $H^{1/2}$-norm is defined by
\[\Vert x \Vert_{H^{1/2}}^2=|x_0|^2+2\pi \sum_{k\in \Z} |k| |x_k|^2.\]
The Hilbert space $E$ has an orthogonal decomposition $E=E^+\oplus E^0\oplus E^-$, 

 A Hamiltonian $H\in C^{\infty}(\R^{2n})$ is said to be admissible for $Y$ if
\begin{itemize}
\item The zero set $\op{supp}(H)=\{x\in\R^{2n}\mid H(x)=0\}$ is compact, $\partial \op{supp}(H)$ is smooth.
\item The radial vector field is transverse to $\partial \op{supp}(H)$.
\item $Y\subset \op{supp}(H)$.
\item There exists $r_0,a$ such that for all $x$ with $\|x\|>r_0$, $H(x)=a\|x\|^2$.
\end{itemize}
The $a$ in the above definition is called the {\bf slope} of $H$.

For $x\in E$, we define
\[
	\Ac_{H}(x) = -\tfrac{1}{2}\int_{0}^{1}{J\dot{x}(t)\cdot x(t) dt} - \int_{0}^{1}{H\big(t,x(t)\big)dt}.
\]

For an $S^1$ invariant subset $X\subset E$, Fadell and Rabinowitz have defined an index $\FR(X)\in\Z$ as follows. Consider the classifying map $f:X\times_{S^1} ES^1\to BS^1=\CP^{\infty}$. So $f$ induces a map $f^*:H^*(\CP^\infty)\to H^*_{S^1}(X)=H^*(X\times_{S^1} ES^1)$. Let $u$ be a generator of the ring $H^*(\CP^\infty)$ and define
\[\FR(X)=\max\{k\in\Z|f^*u^k\neq 0\}.\]
%

Ekeland and Hofer defined a subgroup $\Gamma$ of the group of homeomorphisms of $E$ with compact support (which we recall later) and denoted by $S^+$ the unit sphere in $E^+$. For an $S^1$-invariant subspace $\xi\subset E$, they also defined
\[\ind(\xi)=\min_{h\in\Gamma}\FR(X\cap h(S^+))\]
Finally, they defined
\[c_k^{EH}(H)=\inf\left\{\sup\Ac_{H}(\xi)\big|\ind(\xi)\geq k\right\}.\]

Let us now recall the definition of the group $\Gamma$: a homeomorphism $h$ belongs to $\Gamma$ if $h$ is of the following form:
\[
	h(x)=e^{\gamma^{+}(x)}x^{+} + x^{0} + e^{\gamma^{-}(x)}x^{-} + K(x)
\]
where $\gamma^{+}$ and $\gamma^{-}$ are maps $\Hi\to\R$ which are required to be continuous, $S^{1}$-invariant and mapping bounded sets into bounded sets while $K:\Hi\to\Hi$ is continuous, $S^{1}$-equivariant, mapping bounded sets to pre-compact sets. Additionally, there must exist a number $\rho>0$ such that, either  $\Ac_H(x)\leq0$ or $\|x\|\geq\rho$ implies $\gamma^{+}(x)=\gamma^{-}(x)=0$ and $K(x)=0$.

%

\subsection{The equivariant symplectic homology capacities}\label{sec:GH}
For this section we fix a star-shaped domain $X$ with smooth boundary and an admissible Hamiltonian $H$ of slope $a$. By admissible we mean that $H$ is $C^2$-small in $X$ and $H(x)=a \|x\|^2$ for $\|x\|$ large.

In this situation, one can define an $S^1$-equivariant Floer homology of $H$, see section \ref{SHS1}. It satisfies the following properties:
\begin{proposition}
\label{prop:ch1}
The positive $S^1$-equivariant Floer homology $FH^{S^1,+}(H)$ has the following properties:
\begin{description}
\item{(Action filtration)}
For each $L\in\R$, there is a $\Q$-module $FH^{S^1,+,L}(H)(X,\lambda,\Gamma)$ which is an invariant of $H$. If $L_1<L_2$, then there is a well-defined map
\begin{equation}
\label{eqn:il1l2}
\imath_{L_2,L_1}: FH^{S^1,+,L_1}(H) \longrightarrow FH^{S^1,+,L_2}(H).
\end{equation}
These maps form a directed system, and we have the direct limit
\[
\lim_{L\to\infty} FH^{S^1,+,L}(H) = FH^{S^1,+}(H).
\]
We denote the resulting map $FH^{S^1,+,L}(H) \to FH^{S^1,+}(H)$ by $\imath_L$. 
\item{($U$ map)} There is a distinguished map
\[
U:FH^{S^1,+}(H)\longrightarrow FH^{S^1,+}(H),
\]
which respects the action filtration in the following sense: For each $L\in\R$ there is a map
\[
U_L:FH^{S^1,+,L}(H) \longrightarrow FH^{S^1,+,L}(H).
\]
If $L_1<L_2$ then $U_{L_2}\circ \imath_{L_2,L_1} = \imath_{L_2,L_1}\circ U_{L_1}$. The map $U$ is the direct limit of the maps $U_L$, i.e.
\begin{equation}
\label{eqn:iLU}
\imath_L\circ U_L = U\circ \imath_L.
\end{equation}
\item{(Reeb Orbits)}
If $L_1<L_2$, and if there does not exist a Hamiltonian orbit $\gamma$  with action $\mc{A}(\gamma)\in (L_1,L_2]$, then the map \eqref{eqn:il1l2} is an isomorphism.
\item{($\delta$ map)} There is a distinguished map
\[
\delta: FH^{S^1,+}(H) \longrightarrow H_*(X,\partial X;\Q) \tensor H_*(BS^1;\Q).
\]
\end{description}
\end{proposition}

\begin{definition}
\label{def:cknondeg}
Let $(X,\lambda)$ be a nondegenerate Liouville domain and let $k$ be a positive integer. Define
\[
c_k^{GH}(H)\in(0,\infty]
\]
to be the infimum over $L$ such that there exists $\alpha\in FH^{S^{1},+,L}(H)$ satisfying
\begin{equation}
\label{eqn:alpha}
\delta U^{k-1}\imath_L\alpha = [X]\tensor [\op{pt}] \in H_*(X,\partial X)\tensor H_*(BS^1).
\end{equation}
\end{definition}



\subsection{The idea of the Proof of Theorem \ref{thmfancystatement}}
Note that capacities are usually defined for \emph{nice star-shaped domains} i.e. star-shaped domains with smooth non-degenerate boundary and then extended to all star-shaped domains by ``continuity'', see, for instance \cite[Remark 1.3]{GuH}.
Theorem \ref{thmfancystatement} is a consequence of the following result:
\begin{theorem}\label{thm:main}
	Given $X$ a nice star-shaped domain and $H$ an admissible Hamiltonian for $X$, we have for each $k\in\N$,
	\[c_{k}^{EH}(H)=c_k^{CH}(H).\]
\end{theorem}
\noindent We now outline the sequence of results which lead to the proof of Theorem \ref{thm:main}.\\

The first step is to reformulate both sequences of capacities:
\begin{proposition}\label{prop:eh}
	Let $H$ be an admissible Hamiltonian of slope $L>0$. Then for each $k\in\N$,
	\[
		c_k^{EH}(H)=\left\{\begin{array}{ll}\inf\left\{c\mid \op{ind}_{EH}\left(\{\Ac_H\le c\}\right)\ge k\right\},&\text{ if }k\le \op{ind}_{EH}\left(\{\Ac_H\le L\}\right),\\
		\infty,&\text{ if }k> \op{ind}_{EH}\left(\{\Ac_H\le L\}\right).\end{array}\right.
	\]
\end{proposition}
Suppose that $c\le L$ is a regular value of $\Ac_H$. Now let \[\kappa_c:= \max\{j\mid \exists\,\sigma\in\op{Im}\left(FH^{S^1,+,c}_{2j+n-1}(H)\to FH^{S^1,+}(H)\right)\text{ such that }U^j\sigma\neq 0\}.\]
An immediate reformulation of $c_k^{GH}$ is given by:
\begin{proposition}\label{prop:ch}
Let $H$ be an admissible Hamiltonian of slope $L>0$. Then for each $k\in\N$,
\[c_k^{CH}(H)=\left\{\begin{array}{ll}\inf\left\{c\mid \kappa_c(H)\ge k\right\},&\text{ if }k\le \kappa_L(H),\\
\infty,&\text{ if }k> \kappa_L(H).\end{array}\right.\]
\end{proposition}

The main ingredient in proving that $\kappa_c(H)=\op{ind}_{EH}\left(\{\Ac_H\le c\}\right)$ and therefore proving Theorem \ref{thm:main} is the isomorphism of the positive $S^1$-equivariant Morse and Floer homology complexes of $H$. We refer to Sections \ref{HMAH} and \ref{SHS1} for definitions and constructions of those complexes.
\begin{theorem}\label{thm:iso}
Let $H$ be an admissible Hamiltonian. Then for all $c>0$, there exists an isomorphism
\[CM^{S^1,c}(H)\cong CF^{S^1,c}(H).\]
Moreover, it commutes with the $U$ map.
\end{theorem}
The main direct consequence, of Theorem \ref{thm:iso} is the following:
\begin{lemma}\label{lem:equality}
	Let $H$ be an admissible Hamiltonian for $X$. Then, for all $c>0$,
\[ \kappa_c(H)=\op{ind}_{FR}\left(\{\eps\leq\Ac_H\le c\}\right).\]
\end{lemma}

The proof of Theorem \ref{thm:main} then, consists of the two following lemmas:
\begin{lemma}\label{lem:inf}
	Let $H$ be an admissible Hamiltonian for $X$ of slope $L$. Then, for all $c>0$,
\[\op{ind}_{EH}\left(\{\Ac_H\le c\}\right) \leq \op{ind}_{FR}\left(\{\eps\leq\Ac_H\le c\}\right).\]
\end{lemma}
\begin{lemma}\label{lem:sup}
	Let $H$ be an admissible Hamiltonian for $X$ of slope $L$. Then, for all $c>0$,
\[\op{ind}_{EH}\left(\{\Ac_H\le c\}\right) \geq \kappa_c(H).\]
\end{lemma}

\subsection{Organization of the paper}
We shall start, in Section \ref{sec:proofslemma} by proving Proposition \ref{prop:eh}, Lemma \ref{lem:equality} (assuming Theorem \ref{thm:iso}), Lemma \ref{lem:inf} and Lemma \ref{lem:sup}.
The rest of the paper is devoted to the proof of Theorem \ref{thm:iso}.
In Section \ref{sec:floer} we recall the definition of $S^1$-equivariant symplectic homology.
Section \ref{sec:abbondandolo} is devoted to equivariant Morse theory of the action functional. We start by recalling the definition of Morse theory of the action functional then introduce the equivariant version.
The last section (Section \ref{sec:isomorphism}) is the proof of Theorem \ref{thm:iso}.

\subsection{Acknowledgments}
Both authors would like to thank Alberto Abbondandolo for the many enlightening discussions. We warmly thank Helmut Hofer and Alan Weinstein for all the historical context. We also thank Stefan Matijevic for sending us an early version of his paper \cite{Matijevic} about a related result.

	J.G. was partially supported by the ANR LabEx CIMI (grant ANR-11-LABX-0040) within the French State Programme ``Investissements d’Avenir'' and by the ANR COSY (ANR-21-CE40-0002) grant. This material is based upon work supported by the National Science Foundation
under Grant No. DMS-1926686

 V.G.B.R.\ was partially supported by NSF grant DMS-1926686, FAPERJ grant JCNE E-26/201.399/2021 and a Serrapilheira Institute grant.
 
V.G.B.R. and J.G. are grateful for the hospitality of the Institute for Advanced Study, where part of this work was completed.

\section{Proofs of Lemmas}\label{sec:proofslemma}
\subsection{Proposition \ref{prop:eh}}
Let $H$ be an admissible Hamiltonian for $X$. Remark that for all $c$, $\{\Ac_H\leq c\}$ is $S^1$-invariant and $\forall c\leq c'$, $\ind(\{\Ac_H\leq c\})\leq\ind(\{\Ac_H\leq c'\})$.

Note that for $c$ large enough (bigger than the slope of $H$), $\ind(\{\Ac_H\leq c\})$ stabilizes.
Indeed, for $c$ bigger than the limit slope of $H$, there are no critical values. The flow of $\nabla\Ac_H$ is in $\Gamma$ and pushes down everything.
\begin{proof}[Proof of Proposition \ref{prop:eh}]
It follows from the definition of $\ind(\{\Ac_H\leq L\})$ that $c_{H,k}^{EH}=\infty$ if $k>\ind(\{\Ac_H\leq L\})$. Now suppose that $k\leq\ind(\{\Ac_H\leq L\})$ and let $c>0$ such that that $\ind\{\Ac_H\leq c\}=k$. Since $\sup\Ac_H(\{\Ac_H\leq c\})=c$, we conclude that \[c_{k}^{EH}(H)\le \inf\big\{c | \ind\{\Ac_H\leq c\}\ge k\big\}.\]

Now let $\xi\subset\Hi$ such that $\ind(\xi)\ge k$ and $c_{k}^{EH}(H)=\sup\Ac_H(\xi)$. Then $\xi\subset\{\Ac_H\leq c_{k}^{EH}(H)\}$. So $k\leq \ind(\xi)\leq\ind\{\Ac_H\leq c_{k}^{EH}(H)\}$ and hence\[\inf\big\{c | \ind\{\Ac_H\leq c\}\ge k\big\}\le c_{k}^{EH}(H).\]
\end{proof}

\subsection{Lemma \ref{lem:equality}}
We give an alternative description of the Fadell-Rabinowitz index $\FR$. Let $U: H_*^{S^1}(X)\to H_*^{S^1}(X)$ be the standard $U$-map defined by $U(\gamma)=f^{\star}u\cap \gamma$.
\begin{proposition}
	Let $X$ be an $S^{1}$-space. Then
	\[
		\FR(X) =\max \{k\in\Z|U^k\neq 0\}.
	\]
\end{proposition}
\begin{proof}
	We first remark that there is a map ${}^\tau U$ on the dual. We have $H^{*}_{S^1}(X)\simeq \op{Hom}\big(H_{*}^{S^1}(X);\Q\big)$.
	\[
		\xymatrix{
			\op{Hom}\big(H_{*}^{S^1}(X);\Q\big) & \op{Hom}\big(H_{*}^{S^1}(X);\Q\big)\ar[l]_{{}^\tau U} \ar[d]_\simeq\\
			H^{*}_{S^1}(X)\ar[u]_\simeq & H^{*}_{S^1}(X) \ar[l]_{\cup f^\star(u)}
		}
	\]
	This imply that
		\[
			{}^\tau U \simeq \cup f^\star(u)
		\]
	We then have, $U=0$ if and only if ${}^\tau U=0$.
	This is implied by the ``duality'' between $U$ and ${}^\tau U$; i.e. $\langle {}^\tau U a,b\rangle = \langle a, Ub\rangle$.
	
	Thus,
	\[
		\sup\{k | U^k\neq 0\} = \sup\{k | ({}^\tau U)^k\neq0\} = \sup\{k | f^\star u^k\neq0\}
	\]
	Indeed, to see the last equality, note first that $\geq$ is obvious and $\leq$ is because $f^\star u^k\cup1=f^\star u^k$.
\end{proof}
Then, Lemma \ref{lem:equality} comes from Theorem \ref{thm:iso} and that \cite{Abbondandolothesis} for the subsets $\{\eps\leq\Ac_H\le c\}$ Morse homology and ``regular'' homology agrees.
\[
	H^{*}_{S^1}(\{\eps\leq\Ac_H\le c\})\simeq HM^{*}_{S^1}(\{\eps\leq\Ac_H\le c\})
\]

\subsection{Lemma \ref{lem:inf}}
\begin{proof}[Proof of Lemma \ref{lem:inf}]
	Let $h\in\Gamma$ and write $\Ac^{\epsilon,c}:=\{\epsilon\le\Ac_H\le c\}$. It follows from the subadditivity of the Fadell--Rabinowitz index that
\[\begin{aligned}\text{ind}_{FR}((\{\Ac_H\le\epsilon\}\cup \Ac^{\epsilon,c})\cap h(S^+))&\le \text{ind}_{FR}((\{\Ac_H\le\epsilon\}\cup\Ac^{\epsilon,c})\cap h(S^+))\cup \Ac^{\epsilon,c})\\&=\text{ind}_{FR}((\{\Ac_H\le\epsilon\}\setminus \Ac^{\epsilon,c})\cap h(S^+))\cup \Ac^{\epsilon,c})\\ &\le \text{ind}_{FR}(\{\Ac_H<\epsilon\}\cap h(S^+))+\text{ind}_{FR}(\Ac^{\epsilon,c}).\end{aligned}\]
Taking the infimum over $h\in\Gamma$, we obtain
\[\op{ind}_{EH}(\{\Ac_H\le c\})\le \text{ind}_{FR}(\Ac^{\epsilon,c}).\]
\end{proof}

\subsection{Lemma \ref{lem:sup}}
This proof is in 4 steps. The key idea is to use \cite[Proposition 1]{EH2} which is the following:
\begin{prop}[\cite{EH2}]
	Let $V$ be a finite dimensional $S^1$-invariant subspace of $E^+$, Then
	\[
		\EH(E^-\oplus E^0\oplus V)=\frac{1}{2}\dim(V)
	\]
\end{prop}
The proof of Lemma \ref{lem:sup} consists in finding a $\kappa_c(H)$-dimensional $S^1$-invariant subspace of $E^+$ which lies in $\{\Ac_H\le c\}$.
\subsubsection*{Step 1}
Fix $\kappa\in\mathbb{N}$ and write $\kappa=n (l-1) +p$, where $p,l\in\Z$ and $1\le p\le n$. Fix $a_1<\dots<a_n$, such that
\[a_1<\dots<a_n<2a_1<\dots<2 a_n<\dots<l a_p.\] Let $B=E(a_1,\dots,a_n)$ and
$a= l a_p$. Let 
\[E_{\kappa}^+=\left\{x=\sum_{k\ge 1} e^{2\pi i k t}x_k\in E^+\mid x_{l}\in \C^{p}\text{ and } x_j=0\text{ for }j>l\right\}.\]
Note that $\text{dim}_{\C} (E_\kappa^+)=n (l-1) +p=\kappa$. Now define $E_\kappa=E_\kappa^+\oplus E^0\oplus E^-$.

	The Hamiltonian $H_B$ can be taken to satisfy the following conditions:
	\[
		\begin{cases}
			0\leq H_{B}(x)<\epsilon, &\text{ for }x\in B\\
			H_B(x) = a\pi\left(\frac{|z_1|^2}{a_1}+\dots+\frac{|z_n|^2}{a_n}\right)-a,& \text{ otherwise}.
		\end{cases}
	\]
	Let $\widetilde{H}(z)=a\pi\left(\frac{|z_1|^2}{a_1}+\dots+\frac{|z_n|^2}{a_n}\right)-a$. So $\widetilde{H}\le H_B$. If $x=\sum_{k\in\Z}e^{2\pi i k t}x_k$, we write $x_k=(x_k^1,\dots,x_k^n)\in\C^n$. For $x\in E_\kappa$,
	\begin{align*}
		\Ac_{H_B}(x)	&= \frac{1}{2}\left(\|x^+\|^2_{H^{\frac{1}{2}}}-\|x^-\|^2_{H^{\frac{1}{2}}}\right)-\int H(x(t))\,dt\\
		&\leq  \frac{1}{2}\|x^+\|^2_{H^{\frac{1}{2}}}-\int_0^1 \widetilde{H}(x(t))\,dt\\
		&=\pi \sum_{k\ge 1}\sum_{j=1}^n k|x_k^j|^2-\left(a \pi \sum_{k\in\Z}\sum_{j=1}^n\frac{ |x_k^j|^2}{a_j}-a\right)\\
		&\leq \pi \sum_{k=1}^l\sum_{j=1}^n\frac{l a_p}{a_j}|x_k^j|^2-\left(a \pi \sum_{k=1}^l\sum_{j=1}^n\frac{ |x_k^j|^2}{a_j}\right)+a\\
		&= a.
	\end{align*}
So $E_\kappa\subset \{\Ac_{H_B}\le a\}$.

\subsubsection*{Step 2}
Let $\kappa$ be the natural number defined as
\[\kappa:=\max\left\{j\,|\,\exists \sigma\in FH^{S^1,+}_{2j+n-1}(H)\text{ such that }U^j\sigma\neq 0\right\}(=\kappa_L).\]
Note that $\Ac_H(x)<L$ for all generators $x$ of $CF^{S^1,+}(H)$. So $c_\kappa(H)<L$. Now choose $a_1<\dots<a_n$ such that $B=E(a_1,\dots,a_n)\subset X$ and $a_n-a_1$ is very small as in Step 1. Now let $a\in\R$ be defined as in Step 1. Note that $c_\kappa(B)=a$. Therefore
\begin{equation}\label{eq:al}a=c_\kappa (B)\le c_\kappa(X)<L.\end{equation}
Therefore \[E_\kappa\subset \{A_{H_B}<L\}.\]
\subsubsection*{Step 3}
It follows from \eqref{eq:al} that there exists a constant $M$ such that $H+M> H_{B}$. So for every $x\in E$, 
\[
	\Ac_H(x)-M=\Ac_{H+M}(x)\leq\Ac_{H_{B}}(x)
\]
Therefore \[E_\kappa\subset \{\Ac_{H}<L+M\}.\]

Since $\Ac_H$ does not have critical values in $[L,L+M]$, we can find $h\in\Gamma$ which is a truncated flow of $-\nabla\Ac_H$ such that \[h\left(\{\Ac_{H}<L+M\}\right)\subset \{\Ac_H\le L\}.\]
Therefore
\[\kappa=\op{ind}_{EH}\left(h(E_\kappa)\right)\le\op{ind}_{EH}\left( h\left(\{\Ac_{H}<L+M\}\right)\right)\le \op{ind}_{EH}\left( \{\Ac_H\le L\}\right).\]

\subsubsection*{Step 4}
For each $c\le L$ which is not a critical value of $\Ac_H$, let \[\kappa_c:= \max\{j\mid \exists\,\sigma\in\op{Im}\left(FH^{S^1,+,c}_{2j+n-1}(H)\to FH^{S^1,+}(H)\right)\text{ such that }U^j\sigma\neq 0\}.\]
Then

\[\op{ind}_{EH}\left(\Ac_H\le c\}\right)\geq \kappa_c.\]
\begin{proof}
	We pick an admissible Hamiltonian for our star-shaped domain, $\widetilde{H}$ which coincides with $H$ ``up to slope'' $c$ and has slope $c+\delta$ for $\delta$ very small. In particular, $\widetilde{H}\le H$. Then, we have
 \[
 	\{\Ac_{\widetilde{H}}\leq c\}\subset\{\Ac_{{H}}\leq c\}
 \]
Since the $U$ map commutes with the continuation map, $\kappa_c$ is the $\kappa$ defined in Step 2 for $\widetilde{H}$. So it follows from Step 3 that
\[\kappa_c\le \op{ind}_{EH}\left(\{\Ac_{\widetilde{H}}\leq c\}\right)\le \op{ind}_{EH}\left(\{\Ac_{{H}}\leq c\}\right).\]
\end{proof}

\section{Equivariant symplectic homology}\label{sec:floer}
(Positive) symplectic homology was developed by Viterbo \cite{V}, using works of Cieliebak, Floer, and Hofer \cite{FH,CFH}. The $S^1$-equivariant version of (positive) symplectic homology was originally defined by Viterbo \cite{V}, and an alternate definition using family Floer homology was given by Bourgeois-Oancea \cite[\S2.2]{bo}, following a suggestion of Seidel \cite{Seidel}. We will use the family Floer homology definition here which is often amenable to computations. We follow the treatment in \cite{gutt, GuH}, with some minor tweaks which do not affect the results as well as some generalisations from \cite{Groman23, O}.

(Positive) symplectic homology is usually defined for Liouville domains but we shall restrict our attention here to star-shaped domains in $\R^{2n}$.
Let $X$ be a star-shaped domains in $\R^{2n}$, so that $X$ is a compact smooth manifold with boundary transverse to the radial vector field and $\lambda_0=\frac{1}{2}\sum_{i=1}^{n}x_idy_i-y_idx_i$ has the properties that $d\lambda_0$ is non-degenerate and that $\lambda_0\big|_{\partial X}$ is a contact form.  We  say that $X$ is \emph{non-degenerate} if the linearized return map of the Reeb flow at each closed Reeb orbit on $\partial X$, acting on the contact hyperplane $\ker\lambda$, does not have $1$ as an eigenvalue.

 In this situation, for each $L\in \R$ we have an $L$-filtered positive  $S^1$-equivariant symplectic homology, $SH^{S^1,+,L}(X)$, which will be defined properly in \S\ref{sec:SHS1+}. To simplify notation, we often denote $SH^{S^1,+,L}(X)$ by $CH^{L}(X)$ below\footnote{The reason for this notation is that positive $S^1$-equivariant symplectic homology can be regarded as a substitute for linearized contact homology, which can be defined without transversality difficulties \cite[\S3.2]{bo}.}. These are $\mathbb{Q}$-vector\footnote{It is also possible to define positive $S^1$-equivariant symplectic homology with integer coefficients. However the torsion in the latter is not relevant to the applications explained here, and it will simplify our discussion to discard it.} spaces that come equipped with maps $\imath_{L_1,L_2}\co CH^{L_1}(X)\to CH^{L_2}(X)$ for $L_1\leq L_2$ such that $\imath_{L,L}$ is the identity and $\imath_{L_2,L_3}\circ \imath_{L_1,L_2}=\imath_{L_1,L_3}$.\footnote{Warning: In \cite{GuH} the map that we denote by $\imath_{L_1,L_2}$ is denoted by $\imath_{L_2,L_1}$.} The $CH^L(X)$ are $\Z$-graded.
The (unfiltered) positive $S^1$-equivariant symplectic homology of $(X)$ is $CH(X)=\varinjlim_L CH^L(X)$ where the direct limit is constructed using the maps $\imath_{L_1,L_2}$.

\subsection{Symplectic homology}\label{sec:symplectichomology}

%
Consider a $1$-periodic Hamiltonian on $\R^{2n}$, i.e.\ a smooth function
\[
H:S^1\times\R^{2n}\longrightarrow\R
\]
where $S^1=\R/\Z$. Such a function $H$ determines a vector field $X_H^\theta$ on $\R^{2n}$ for each $\theta\in S^1$, defined by ${\omega_0}(X_H^\theta,\cdot) = dH(\theta,\cdot)$. Let $\Per(H)$ denote the set of $1$-periodic orbits of $X_H$, i.e.\ smooth maps $\gamma:S^1\to\R^{2n}$ satisfying the equation $\gamma'(\theta) = X_H^\theta\big(\gamma(\theta)\big)$.

We now fix our non-degenerate star-shaped domain $X$ with smooth boundary $Y$.
Recall that the completion $(\widehat{X},\widehat{\lambda})$ of $(X,\lambda_0)$ is defined by
\[
\widehat{X}:=X\cup\bigl([0,\infty)\times Y\bigr) \quad\textrm{and}\quad\widehat{\lambda_0}:=
\begin{cases}
	\lambda_0 &\textrm{on } X,\\
	e^\rho \lambda_0\big|_Y &\textrm{on }[0,\infty)\times Y
\end{cases}
\]
where $\rho$ denotes the $[0,\infty)$ coordinate. We can, of course, identify here $(\widehat{X}, \widehat{\lambda_0})$ with $(\R^{2n},\lambda_0)$.
\begin{definition}
\label{def:Hstd}
An \textbf{ admissible Hamiltonian for $X$\/} is a smooth function $H : S^{1}\times\R^{2n} \rightarrow \R$ satisfying the following conditions:
\begin{description}
\item[(1)]
The restriction of $H$ to $S^1\times X$ is negative, autonomous (i.e.\ $S^1$-independent), and $C^2$-small (so that there are no non-constant 1-periodic orbits). Furthermore,
\begin{equation}
\label{eqn:HSpec}
H > -\epsilon
\end{equation}
on $S^1\times X$.
\item[(2)]
There exists $\rho_{1} \geq 0$ such that on $S^1 \times [\rho_1,\infty]\times Y$ we have
\begin{equation}
\label{eqn:limitingslope}
H(\theta,x) = \beta \|x\|^2 + \beta'
\end{equation}
with $0<\beta\notin\Spec(Y,\lambda)\cup2\pi\Z$  and $\beta'\in\R$. The constant $\beta$ is called the \textbf{ limiting slope\/} of $H$.
\item[(3)]
	There exists $0\leq\rho_{0} \leq \rho_1$ such that on $S^1 \times [\rho_0,\rho_1-\eps]\times Y$ we have
\begin{equation}
\label{eqn:limitingslope2}
H(\theta,\rho,y) = \beta e^\rho + \beta''
\end{equation}
with $\beta''\in\R$. 
\item[(4)]
There exists a small, strictly convex, increasing function $h:[1,e^{\rho_0}]\to\R$ such that on $S^1\times[0,\rho_0]\times Y$, the function $H$ is $C^{2}$-close to the function sending $(\theta,\rho,x)\mapsto h(e^\rho)$. The precise sense of ``small'' and ``close'' that we need here is explained in Remarks~\ref{rem:orbitsofHstand} and \ref{rem:hsmall}.
\item[(5)]
The Hamiltonian $H$ is nondegenerate, i.e.\ all $1$-periodic orbits of $X_{H}$ are nondegenerate.
\end{description}
We denote the set of admissible Hamiltonians by $\Hstd$.
\end{definition}

\begin{remark}
\label{rem:orbitsofHstand}
Condition (1) implies that the only $1$-periodic orbits of $X_H$ in $X$ are constants; they correspond to critical points of $H$.

The significance of condition (2) and (3) is as follows. On $S^{1}\times [0,\infty)\times Y$, for a Hamiltonian of the form $H_1(\theta,\rho,y)=h_1(e^{\rho})$, we have
\[
X_{H_1}^{\theta}(\rho,y) = -h'_1(e^{\rho})R_{\lambda}(y).
\]
Hence for such a Hamitonian $H_1$ with $h_1$ increasing, a $1$-periodic orbit of $X_{H_1}$ maps to a level $\{\rho\}\times Y$, and the image of its projection to $Y$ is the image of a (not necessarily simple) periodic Reeb orbit of period $h'_1(e^{\rho})$. In particular, condition (2) and (3) implies that there is no $1$-periodic orbit of $X_H$ in $[\rho_0,\rho_1-\eps) \times Y\cup[\rho_1,\infty)\times Y$.

Condition (4) ensures that for any non-constant $1$-periodic orbit $\gamma_H$ of $X_H$, there exists a (not necessarily simple) periodic Reeb orbit $\gamma$ of period $T<\beta$ such that the image of $\gamma_H$ is close to the image of $\gamma$ in $\{ \rho \}\times  Y$ where  $T=h'(e^{\rho})$.
\end{remark}

\begin{definition}
\label{def:admJ}
An $S^1$-family of almost complex structures $J:S^1\to\End(T\widehat{X})$ is \textbf{ admissible\/} if it satisfies the following conditions:
\begin{itemize}
\item
$J^\theta$ is ${\omega_0}$-compatible for each $\theta\in S^1$.
\item
There exists $\rho_2\ge 0$ such that on $[\rho_2,\infty0\times Y$, the almost complex structure $J^\theta$ does not depend on $\theta$, is invariant under translation of $\rho$, sends $\xi$ to itself compatibly with $d\lambda$, and satisfies
\begin{equation}
\label{eqn:JReeb}
J^\theta (\partial_{\rho})=R_{\lambda}.
\end{equation}
\end{itemize}
We denote the set of all admissible $J$ by $\mathcal{J}$.
\end{definition}

Given $J\in\mathcal{J}$, and $\gamma_-,\gamma_+\in\Per(H)$, let $\widehat{\M}(\gamma_-,\gamma_+;J)$ denote the set of maps
\[
u: \R\times S^1\longrightarrow \widehat{X}
\]
satisfying Floer's equation
\begin{equation}
\label{eqn:Floer}
\frac{\partial u}{\partial s}(s,\theta) + J^{\theta}\bigl(u(s,\theta)\bigr)\biggl(\frac{\partial u}{\partial \theta}(s,\theta) - X_H^{\theta}\bigl(u(s,\theta)\bigr)\biggr)=0
\end{equation}
as well as the asymptotic conditions
\[
\lim_{s\to\pm\infty}u(s,\cdot) = \gamma_\pm.
\]
If $J$ is generic and $u\in\widehat{\M}(\gamma_-,\gamma_+;J)$, then $\widehat{\M}(\gamma_-,\gamma_+;J)$ is a manifold near $u$ whose dimension is the Fredholm index of $u$ defined by
\[
\op{ind}(u) = \op{CZ}_\tau(\gamma_+) - \op{CZ}_\tau(\gamma_-).
\]
Here $\op{CZ}_\tau$ denotes the Conley-Zehnder index computed using trivializations $\tau$ of $\gamma_\pm^{\star}T\widehat{X}$ that extend to a trivialization of $u^{\star}T\widehat{X}$. Note that $\R$ acts on $\widehat{\M}(\gamma_-,\gamma_+;J)$ by translation of the domain; we denote the quotient by $\M(\gamma_-,\gamma_+;J)$.

\begin{definition}	
Let $H\in\Hstd$, and let $J\in\mathcal{J}$ be generic.
Define the Floer chain complex $(CF(H,J),\partial)$ as follows. The chain module $CF(H,J)$ is the free $\Q$-module\footnote{It is also possible to use $\Z$ coefficients here, but we will use $\Q$ coefficients in order to later establish the Reeb Orbits property in Proposition~\ref{prop:ch}, which leads to the Reeb Orbits property of the capacities $c_k$. In special cases when the Conley-Zehnder index of a $1$-periodic orbit is unambiguously defined, for example when all $1$-periodic orbits are contractible and $c_1(TX)|_{\pi_2(X)}=0$, the chain complex is graded by minus the Conley-Zehnder index.} generated by the set of $1$-periodic orbits $\Per(H)$. If $\gamma_-,\gamma_+\in\Per(H)$, then the coefficient of $\gamma_+$ in $\partial\gamma_-$ is obtained by counting Fredholm index $1$ points in $\M(\gamma_-,\gamma_+;J)$ with signs determined by a system of coherent orientations as in \cite{FH2}. (The chain complexes for different choices of coherent orientations are canonically isomorphic.)
\end{definition}
	
Let $HF(H,J)$ denote the homology of the chain complex $(CF(H,J),\partial)$. Given $H$, the homologies for different choices of generic $J$ are canonically isomorphic to each other, so we can denote this homology simply by $HF(H)$. 

The construction of the above canonical isomorphisms is a special case of the following more general construction. Given two admissible Hamiltonians $H_1,H_2\in\Hstd$, write $H_1\le H_2$ if $H_1(\theta,x)\le H_2(\theta,x)$ for all $(\theta,x)\in S^1\times\widehat{X}$. In this situation, one defines a \emph{continuation morphism} $HF(H_1)\to HF(H_2)$ as follows; cf.\ \cite[Thm.\ 4.5]{gutt} and the references therein. Choose generic $J_1,J_2\in\mathcal{J}$ so that the chain complexes $CF(H_i,J_i)$ are defined for $i=1,2$. Choose a generic homotopy $\{(H_s,J_s)\}_{s\in\R}$ such that $H_s$ satisfies equation \eqref{eqn:limitingslope} for some $\beta,\beta'$ depending on $s$; $J_s\in\mathcal{J}$ for each $s\in\R$; $\partial_sH_s\ge 0$; $(H_s,J_s)=(H_1,J_1)$ for $s<<0$; and $(H_s,J_s)=(H_2,J_2)$ for $s>>0$. One then defines a chain map $CF(H_1,J_1)\to CF(H_2,J_2)$ as a signed count of Fredholm index $0$ maps 
$u:\R\times S^1\rightarrow \widehat{X}$ satisfying the equation
\begin{equation}
\label{eq:floerparam}
		\frac{\partial u}{\partial s} + J_s^{\theta}\circ u\Bigl(\frac{\partial u}{\partial \theta} - X^\theta_{H_s}\circ u\Bigr)=0
\end{equation}
and the asymptotic conditions $\lim_{s\to-\infty}u(s,\cdot) = \gamma_1$ and $\lim_{s\to\infty}u(s,\cdot) = \gamma_2$. The induced map on homology gives a well-defined map $HF(H_1)\to HF(H_2)$. If $H_2\le H_3$, then the continuation map $HF(H_1)\to HF(H_3)$ is the composition of the continuation maps $HF(H_1)\to HF(H_2)$ and $HF(H_2)\to HF(H_3)$. 

\begin{definition}
We define the \emph{symplectic homology} of $(X,\lambda)$ to be the direct limit 
\[
SH(X,\lambda) := \lim_{\substack{\longrightarrow \\ H\in \mathcal{H}_{\textrm{adm}}}}HF(H)
\]
with respect to the partial order $\le$ and continuation maps defined above.
\end{definition}

\subsection{Positive symplectic homology}
\label{sec:psh}

Positive symplectic homology is a modification of symplectic homology in which constant $1$-periodic orbits are discarded.

To explain this, let $H:S^1\times\R^{2n}\rightarrow \R$ be a Hamiltonian in $\Hstd$. The \emph{Hamiltonian action} functional $\mathcal{A}_{H}: C^{\infty}(S^{1},\R^{2n})\rightarrow \R$ is defined by
\[
	\mathcal{A}_{H}(\gamma) := -\int_{S^{1}}\gamma^{\star}{{\lambda_0}} - \int_{S^{1}}H\bigl(\theta,\gamma(\theta)\bigr)d\theta.
\]
If $J\in\mathcal{J}$, then the differential on the chain complex $(CF(H,J),\partial)$ decreases the Hamiltonian action $\mathcal{A}_H$. As a result, for any $L\in\R$, we have a subcomplex $CF^{\le L}(H,J)$ of $CF(H,J)$, generated by the $1$-periodic orbits with Hamiltonian action less than or equal to $L$. 

To see what this subcomplex can look like, note that the $1$-periodic orbits of $H \in \Hstd$ fall into three classes: (i) constant orbits corresponding to critical points in $X$,  (ii) non-constant orbits contained in $[0,\rho_0]\times Y$, and (iii) non-constant orbits contained in $[\rho_1-\eps,\rho_1]\times Y$.

If $x$ is a critical point of $H$ on $X$, then the action of the corresponding constant orbit is equal to $-H(x)$. By \eqref{eqn:HSpec}, this is less than $\epsilon$.

By Remark~\ref{rem:orbitsofHstand}, a non-constant $1$-periodic orbit of $X_H$ is close to a $1$-periodic orbit of $-h'(e^\rho)R_{\lambda}$ located in $\{ \rho \}\times Y$ for  $\rho\in[0,\rho_0]$ with $h'(e^\rho)\in\Spec(Y,\lambda)$. The Hamiltonian action of the latter loop is given by 
\begin{equation}
\label{eqn:actionsclose}
-\int_{S^1}e^\rho\lambda(-h'(e^\rho)R_{\lambda})d\theta -\int_{S^1}h(e^\rho)d\theta=e^\rho h'(e^\rho)-h(e^\rho).
\end{equation}
Since $h$ is strictly convex, the right hand side is a strictly increasing function of $\rho$.

Note also, by \cite{V} , and taking $\rho_1$ sufficiently large, the orbits in (iii) have negative action.
\begin{remark}
\label{rem:hsmall}
In Definition~\ref{def:Hstd}, we assume that $h$ is sufficiently small so that the right hand side of \eqref{eqn:actionsclose} is close to the period $h'(e^\rho)$, and in particular greater than $\epsilon$. We also assume that $H$ is sufficiently close to $h(e^\rho)$ on $S^1\times[0,\rho_0]\times Y$ so that the Hamiltonian actions of the $1$-periodic orbits are well approximated by the right hand side of \eqref{eqn:actionsclose}, so that:
\begin{description}
\item[(i)] The Hamiltonian action of every 1-periodic orbit of $X_H$ corresponding to a critical point on $X$ is less than $\epsilon$; and the Hamiltonian action of every other $1$-periodic orbit is greater than $\epsilon$.
\item[(ii)]
If $\gamma$ is a Reeb orbit of period $T<\beta$, and if $\gamma'$ is a $1$-periodic orbit of $X_H$ in $[0,\rho_0]\times Y$ associated to $\gamma$, then
\[
|\mc{A}_H(\gamma') - T | < \min\left\{\beta^{-1},\tfrac{1}{3}\op{gap}(\beta)\right\}.
\]
Here $\op{gap}(\beta)$ denotes the minimum difference between two elements of $\Spec(Y,\lambda)$ that are less than $\beta$. 
\end{description}
\end{remark}

We can now define positive symplectic homology.

\begin{definition}
\label{def:psh}
	Let $(X,\lambda)$ be a Liouville domain, let $H$ be a Hamiltonian in $\Hstd$, and let $J\in\mathcal{J}$.

Consider the quotient complex
\[
CF^+(H,J) \eqdef \frac{CF(H,J)}{CF^{\le\epsilon}(H,J)}.
\]
The homology of the quotient complex is independent of $J$, so we can denote this homology by $HF^+(H)$. More generally, if $H_1\le H_2$, then the chain map used to define the continuation map $HF(H_1)\to HF(H_2)$ descends to the quotient, since the Hamiltonian action is nonincreasing along a solution of \eqref{eq:floerparam} when the homotopy is nondecreasing. Thus we obtain a well-defined continuation map $HF^+(H_1)\to HF^+(H_2)$ satisfying the same properties as before.

We now define the \emph{positive symplectic homology} of $(X,\lambda)$ to be the direct limit
\[
SH^{+}(X,\lambda) := \lim_{{\substack{\longrightarrow\\ {H\in \Hstd}}}} HF^{+}(H).
\]
\end{definition}

\subsection{$S^1$-equivariant symplectic homology}
\label{SHS1}

Let $(X,\lambda_0)$ be a star-shaped domain in $\R^{2n}$ with boundary $Y$.
We now review how to define the $S^1$-equivariant symplectic homology $SH^{S^1}(X,\lambda)$.

The $S^1$-equivariant symplectic homology $SH^{S^1}(X,\lambda_0)$ is defined as a limit as $N\to\infty$ of homologies $SH^{S^1,N}(X,\lambda_0)$, where $N$ is a nonnegative integer. To define the latter, fix the perfect Morse function $f_N:\C P^N\to\R$ defined by
\[
f_N\bigl([w^0:\ldots:w^n]\bigr)=\frac{\sum_{j=0}^Nj|w^j|^2}{\sum_{j=0}^N|w^j|^2}.
\]
Let $\widetilde{f}_N:S^{2N+1}\to\R$ denote the pullback of $f_N$ to $S^{2N+1}$. We will consider gradient flow lines of $\widetilde{f_N}$ and $f_N$ with respect to the standard metric on $S^{2N+1}$ and the metric that this induces on $\C P^N$.

\begin{remark}\label{rmk:periodicity}
The family of functions $f_N$ has the following two properties which are needed below. We have two isometric inclusions $i_0,i_1:\C P^N\to\C P^{N+1}$ defined by $i_0([z_0:\ldots:z_{N}]) = [z_0:\ldots:z_{N}:0]$ and $i_1([z_0:\ldots:z_{N}])=[0:z_0:\ldots:z_{N}]$. Then:
\begin{description}
\item[(1)]
The images of $i_0$ and $i_1$ are invariant under the gradient flow of $f_{N+1}$.
\item[(2)]
We have $f_N=f_{N+1}\circ i_0=f_{N+1}\circ i_1 + \op{constant}$, so that the gradient flow of $f_{N+1}$ pulls back via $i_0$ or $i_1$ to the gradient flow of $f_N$.
\end{description}
\end{remark}

Now choose a ``parametrized Hamiltonian''
\begin{equation}
\label{eqn:Hamparam}
H:S^1\times \R^{2n}\times S^{2N+1}\longrightarrow \R
\end{equation}
which is $S^1$-invariant in the sense that
\[ 
H(\theta+\varphi,x,\varphi z) = H(\theta,x,z)\qquad \forall \theta,\varphi\in S^1=\R/\Z,\; x\in\R^{2n},\; z\in S^{2N+1}.
\]
Here the action of $S^1=\R/\Z$ on $S^{2N+1}\subset \C^{N+1}$ is defined by $\varphi\cdot z=e^{2\pi i\varphi}z$.

\begin{definition}
\label{def:aph}
A parametrized Hamiltonian $H$ as above is \emph{admissible} if:
\begin{description}
\item[(i)]
For each $z\in S^{2N+1}$, the Hamiltonian 
\[
H_z = H(\cdot,\cdot,z):S^1\times\widehat{X}\longrightarrow\R
\]
satisfies conditions (1), (2), (3), and (4) in Definition~\ref{def:Hstd}, with $\beta$ and $\beta'$ independent of $z$.
\item[(ii)]
If $z$ is a critical point of $\widetilde{f}_N$, then the $1$-periodic orbits of $H_z$ are nondegenerate.
\item[(iii)] $H$ is nondecreasing along downward gradient flow lines of $\widetilde{f}_N$.
\end{description}
\end{definition}

Let $\Per^{S^1}(H, \tilde{f}_N)$ denote the set of pairs $(\gamma,z)$, where $z\in S^{2N+1}$ is a critical point of $\tilde{f}_N$, and $\gamma$ is a $1$-periodic orbit of the Hamitonian $H_z$.
Note that $S^1$ acts freely on the set $\Per^{S^1}(H, \tilde{f}_N)$ by
\[
\varphi\cdot (\gamma,z) = \big(\gamma(\cdot - \varphi), \varphi\cdot z\big).
\]
If $p=(\gamma,z)\in\Per^{S^1}(H,\tilde{f}_N)$, let $S_p$ denote the orbit of $(\gamma,z)$ under this $S^1$ action.

Next, choose a generic map
\begin{equation}
\label{eqn:Jparam}
J:S^1\times S^{2N+1} \to \mc{J}, \quad (\theta,z) \mapsto J^\theta_z,
\end{equation}
which is $S^1$-invariant in the sense that
\[
J^{\theta+\varphi}_{\varphi\cdot z} = J^\theta_z
\]
for all $\varphi,\theta\in S^1$ and $z\in S^{2N+1}$. 

Let $p^-=(\gamma^-,z^-)$ and $p^+=(\gamma^+,z^+)$ be distinct elements of $\Per^{S^1}(H,\tilde{f}_N)$. Define $\widehat{\Mod}(S_{p^-},S_{p^+};J)$
to be the set of pairs $(u,\eta)$, where $\eta:\R\to S^{2N+1}$ and $u:\R\times S^1\to\widehat{X}$, satisfying the following equations:
\begin{equation}
\label{eq:fleorparam}
\left\{
\begin{aligned}
\partial_su+J^{\theta}_{\eta(s)}\circ u\bigl(\partial_{\theta}u-X_{H_{\eta(s)}^\theta}\circ u\bigr) &=0,\\
\dot{\eta}+\vec{\nabla}\tilde{f}_N(\eta) &=0,\\
\lim_{s\rightarrow\pm\infty}\bigl(u(s,\cdot),\eta(s)\bigr) &\in S_{p^\pm}.
\end{aligned}
\right.
\end{equation}
Here the middle equation is a modification of Floer's equation \eqref{eqn:Floer} which is ``parametrized by $\eta$''. Note that $\R$ acts on the set $\widehat{\Mod}(S_{p^-},S_{p^+};J)$ by reparametrization: if $\sigma\in\R$, then
\[
\sigma\cdot(u,\eta) = \big(u(\cdot-\sigma,\cdot),\eta(\cdot-\sigma)\big).
\]
In addition, $S^1$ acts on the set $\widehat{\Mod}(S_{p^-}, S_{p^+};J)$ as follows: if $\tau\in S^1$, then
\[
\tau\cdot (u,\eta):=\bigl(u(\cdot,\cdot-\tau),\tau\cdot\eta\bigr).
\]
Let $\Mod^{S^1}(S_{p^-},S_{p^+};J)$ denote the quotient of the set $\widehat{\Mod}(S_{p^-},S_{p^+};J)$ by these actions of $\R$ and $S^1$.

If $J$ is generic, then $\Mod^{S^1}(S_{p^-},S_{p^+};J)$ is a manifold near $(u,\eta)$ of dimension
\[
\op{ind}(u,\eta) = (\op{ind}(f_N,z^-) - \op{CZ}_\tau(\gamma^-)) - (\op{ind}(f_N,z^+) - \op{CZ}_\tau(\gamma^+)) - 1.
\]
Here $\op{ind}(f_N,z^\pm)$ denotes the Morse index of the critical point $z^\pm$ of $f_N$, and $\op{CZ}_\tau$ denotes the Conley-Zehnder index with respect to a trivialization $\tau$ of ${(\gamma^\pm)}^{\star}T\widehat{X}$ that extends over $u^{\star}T\widehat{X}$.

\begin{definition}
\label{def:complex}
\cite[\S2.2]{bo}
Define a chain complex $\left(CF^{S^1,N}(H,J),\partial^{S^1}\right)$ as follows. The chain module $CF^{S^1,N}(H,J)$ is the free $\Q$ module\footnote{It is also possible to define $SH^{S^1,+}$, using $\Z$ coefficients, as with $SH$.} generated by the orbits $S_p$.
If $S_{p^-},S_{p^+}$ are two such orbits, then the coefficient of $S_{p^+}$ in $\partial^{S^1}S_{p^-}$ is a signed count of elements $(u,\eta)$ of $\Mod^{S^1}(S_{p^-},S_{p^+};J)$ with $\op{ind}(u,\eta)=1$.
\end{definition}

We denote the homology of this chain complex by $HF^{S^1,N}(H)$. This does not depend on the choice of $J$, by the usual continuation argument; one defines continuation chain maps using a modification of \eqref{eq:fleorparam} in which the second line is replaced by an ``$\eta$-parametrized'' version of Floer's continuation equation \eqref{eq:floerparam}.

We now define a partial order on the set of pairs $(N,H)$, where $N$ is a nonnegative integer and $H$ is an admissible parametrized Hamiltonian \eqref{eqn:Hamparam}, as follows. Let $\widetilde{i}_0:S^{2N+1}\to S^{2N+3}$ denote the inclusion sending $z\mapsto (z,0)$. (This lifts the inclusion $i_0$ defined in Remark~\ref{rmk:periodicity}.) Then $(N_1,H_1)\le (N_2,H_2)$ if and only if:
\begin{itemize}
\item
$N_1\le N_2$, and
\item
$H_1 \le (\widetilde{i_0}^{\star})^{N_2-N_1}H_2$ pointwise on $S^1\times\R^{2n}\times S^{2N_1+1}$.
\end{itemize}
In this case we can define a continuation map $HF^{S^1,N_1}(H_1) \to HF^{S^1,N_2}(H_2)$ using an increasing homotopy from $H_1$ to $(\widetilde{i_0}^{\star})^{N_2-N_1}H_2$ on $S^1\times\widehat{X}\times S^{2N_1+1}$.

\begin{definition}
Define the \emph{$S^1$-equivariant symplectic homology}
\[
SH^{S^1}_*(X,\lambda) := \lim_{\substack{\longrightarrow\\ N,H}}HF^{S^1,N}_*(H).
\]
\end{definition}

\subsubsection{Positive $S^1$-equivariant symplectic homology}
\label{sec:SHS1+}

As for symplectic homology, $S^1$-equivariant symplectic homology also has a positive version in which constant $1$-periodic orbits are discarded.

\begin{definition}
Let $H:S^1\times\R^{2n}\times S^{2N+1}\to\R$ be an admissible parametrized Hamiltonian. The \emph{parametrized action functional} $\mc{A}_H: \ C^{\infty}(S^1,\R^{2n})\times S^{2N+1}\longrightarrow\R$  is defined by
\begin{equation}
\label{eq:paramaction}	\mc{A}_H(\gamma,z):=-\int_{\gamma}\widehat{\lambda}-\int_{S^1}H\bigl(\theta,\gamma(\theta),z\bigr)d\theta.
\end{equation}
\end{definition}

\begin{lemma}
\label{lem:dpa}
If $H$ is an admissible parametrized Hamiltonian, and if $J$ is a generic $S^1$-invariant family of almost complex structures as in \eqref{eqn:Jparam}, then the differential $\partial^{S^1}$ on $CF^{S^1,N}(H,J)$ does not increase the parametrized action \eqref{eq:paramaction}.
\end{lemma}

\begin{proof}
Given a solution $(u,\eta)$ to the equations \eqref{eq:fleorparam}, one can think of $\eta$ as fixed and regard $u$ as a solution to an instance of equation \eqref{eq:floerparam}, where $J_s$ and $H_s$ in \eqref{eq:floerparam} are determined by $\eta$. By condition (iii) in Definition~\ref{def:aph}, this instance of \eqref{eq:floerparam} corresponds to a nondecreasing homotopy of Hamiltonians. Consequently, the action is nonincreasing along this solution of \eqref{eq:floerparam} as before.
\end{proof}

It follows from Lemma~\ref{lem:dpa} that for any $L\in\R$, we have a subcomplex $CF^{S^1,N,\le L}(H,J)$ of $CF^{S^1,N}(H,J)$, spanned by $S^1$-orbits of pairs $(\gamma,z)$ where $z\in\op{Crit}(\tilde{f}_N)$ and $\gamma$ is a $1$-periodic orbit of $H_z$ with $\mc{A}_H(z,\gamma)\le L$.

As in \S\ref{sec:psh}, if the $S^1$-orbit of $(\gamma,z)$ is a generator of $CF^{S^1,N}(H,J)$, then there are two possibilities: (i) $\gamma$ is a constant orbit corresponding to a critical point of $H_z$ on $X$, and $\mc{A}_H(z,\gamma)<\epsilon$; or (ii) $\gamma$ is close to a Reeb orbit in $\{\rho\}\times Y$ with period $-h'(e^\rho)$, and $\mc{A}_H(z,\gamma)$ is close to this period; in particular $\mc{A}_H(\gamma,z)>\epsilon$.

\begin{definition}
\label{def:positiveequivSH}
Consider the quotient complex
\begin{equation}
\label{eqn:peshquotient}
CF^{S^1,N,+}(H,J) := \frac{CF^{S^1,N}(H,J)}{CF^{S^1,N,\leq\epsilon}(H,J)}.
\end{equation}
As in Definition~\ref{def:psh}, the homology of the quotient complex is independent of $J$, so we can denote this homology by $HF^{S^1,N,+}(H)$; and we have continuation maps $HF^{S^1,N_1,+}(H_1)\to HF^{S^1,N_2,+}(H_2)$ when $(N_1,H_1)\le (N_2,H_2)$. We now define the \emph{positive $S^1$-equivariant symplectic homology} by
\begin{equation}
\label{eqn:pesh}
SH^{S^1,+}(X,\lambda) := \lim_{\substack{\longrightarrow\\ {N,H}}}HF^{S^1,N,+}(H).
\end{equation}
\end{definition}

%


%
\section{Equivariant Morse Homology for the action functional}\label{sec:abbondandolo}
\subsection{Morse homology for Hilbert spaces}\label{sec:generalmorse}
Abbondandolo and Majer \cite{AM, AbMa, AbMa2, AbMa3}, have defined a relative Morse homology on Hilbert spaces for some functionals. This applies, in particular, to the (parametrized) action functional but, we start by recalling the general definition, following the aforementionned references as it might be of independent interest. Let $\Hi$ be a real Hilbert space and $L$ a linear, invertible, self-adjoint operator on $\Hi$, one considers the class of functionals $f:\Hi\to\R$ of the form
\[
	f(x) = \frac{1}{2} (Lx\,,\,x)+b(x)
\]
where $b$ is $C^2$ and $\nabla b:\Hi\to\Hi$ is a compact map.
Denote this class of functionals by $\FF(L)$.
The main idea is that under suitable assumptions, even so the Morse indices and coindices of the critical points are infinite, the intersections of stable and unstable manifolds, $W^u(x)\intt W^s(y)$ are finite dimensional. To prove such a result (and to define a relative Morse index) requires a orthogonal decomposition of the Hilbert space $\Hi$ in two subspaces.

Given a bounded self-adjoint operator $S:\Hi\to\Hi$, denote by $V^+(S)$ (respectively $V^-(S)$) the maximal $S$-invariant subspace on which $S$ is strictly positive (respectively strictly negative). The spaces $V^+(S)$ and $V^-(S)$ are called the \textbf{positive eigenspace of $S$} and the \textbf{negative eigenspace of $S$} respectively.
Since the operator $L$ has been fixed, we denote by $\Hi^+$ and $\Hi^-$ the positive and negative eigenspaces of $L$
\[
	\Hi^+\eqdef V^+(L),\qquad\Hi^-\eqdef V^-(L).
\]
Note that we have $\Hi=\Hi^+\oplus\Hi^-$.

The Hessian of a functional $f\in\FF(L)$ at $x$ is given by
\[
	D^{2}f(x)=L+D^2b(x).
\]
Note that $D^{2}f(x)$ is a Fredholm operator since $D^2b(x)$ is a compact linear operator (because $\nabla b$ is compact).

We now recall the notion of  ``relative Morse index'' for the critical points of $f$.
\begin{definition}
	Let $V$ and $W$ be closed linear subspaces of a Hilbert space $\Hi$. They form a \textbf{Fredholm pair} if $\dim(V\intt W)<\infty$, $V+W$ is closed and $\dim\frac{\Hi}{{V+W}} = \dim(V+W)^{\perp}=\dim(V^{\perp}\intt W^{\perp})<\infty$.
\end{definition}
\begin{remark}
	An operator $A:\Hi_1\to\Hi_2$ is Fredholm if and only if $\big(\Hi_1\times\{0\},Graph(A)\big)$ is a Fredholm pair in $\Hi_1\times\Hi_2$. The \textbf{index of a Fredholm pair} $(V,W)$ is defined as 
	\[
		ind(V,W)=\dim(V\intt W)-\codim(V+W)\in\Z.
	\]
\end{remark}
Let $V$ and $W$ be closed linear subspaces of a Hilbert space $\Hi$. $W$ is a \textbf{compact perturbation of $V$} if $P_{W}-P_{V}$ is compact, where $P$ is the orthogonal projection. In particular $(V,W^{\perp})$ is a Fredholm pair. The relative dimension of $V$ with respect to $W$ is defined as $\dim(V,W):=ind(V,W^{\perp})=\dim(V\intt W^{\perp})-\dim(V^{\perp}\intt W)$.

If $A$ is a self-adjoint Fredholm operator and $K$ is a compact operator, $V^{-}(A)$ is a compact perturbation of $V^{-}(A+K)$.

Going back to the functional $f$, we have $D^{2}f(x)=L + D^{2}b(x)$ where $D^{2}b(x)$ is a compact operator. We have that $V^{-}\big(D^{2}f(x)\big)$ is a compact perturbation of $\Hi^{-}$ and we can define the \textbf{relative Morse index} of $x$ as
 \[
 	ind_{\Hi^{-}}(x)=\dim\Big(V^{-}\big(D^{2}f(x)\big),\Hi^{-}\Big).
\]
Remark that when $\Hi^{-}=\{0\}$ , this index is the usual Morse index. We denote by $\op{crit}_{k}(f)$ the set of critical points of $f$ of relative Morse index $k$.

Now, let $x$ and $y$ be critical points of $f$, we look at $W^{u}(x)\cap W^{s}(y)$ to define moduli spaces of gradient trajectories $u'=\nabla f(u)$.

Let $I\subset\R\cup\{-\infty,+\infty\}$ be an interval.
\begin{definition}
	A functional $f\in C^2(\Hi)$ is called \emph{Morse} on $I$ if the Hessian $D^2f(x)$ is invertible for every critical point $x$ such that $f(x)\in I$.
\end{definition}
Assuming that the functional $f$ is Morse, we have the two following facts $\forall p\in W^{u}(x)$:
\begin{enumerate}
	\item $T_{p}W^{u}(x)$ is a compact perturbation of $\Hi^{-}$ with relative dimension $ind_{\Hi^{-}}(x)$
	\item $\big(T_{p}W^{s}(x),\Hi^{-}\big)$ are Fredholm pairs
\end{enumerate}

If $p\in W^{u}(x)\cap W^{s}(y)$, $\big(T_{p}W^{u}(x), T_{p}W^{s}(y)\big)$ is a Fredholm pair of index $ind_{\Hi^{-}}(x)-ind_{\Hi^{-}}(y)$.

In our case, the gradient trajectories are of the form:
\[
	u'(t)=-\nabla f(u) = -Lu-\nabla b
\]
So $u'+Lu=-\nabla b$, multiplying by $e^{tL}$, we have
\[
	\frac{d}{dt}e^{tL}u=e^{tL}(u'+Lu) = -e^{tL}\nabla b
\]
and thus
\[
	u(t)=e^{-tL}\left(u(0)-\int_{0}^{t}{e^{sL}\nabla b\big(u(s)\big)ds}\right).
\]
\begin{definition}
	A functional $f\in C^1(\Hi)$ satisfies the \emph{Palais-Smale condition} on $I$ if every sequence $(x_n)\subset\Hi$ such that $\lim_{n\to\infty}f(x_n)=c\in I$ and $\lim_{n\to\infty}\nabla f(x_n)=0$ is relatively compact.
\end{definition}

\begin{lemma}
	The functional $f$ satisfies the Palais-Smale condition (PS) if and only if all PS sequences are bounded.
\end{lemma}
\begin{proof}
	Indeed, $\nabla f(x)=Lx+\nabla b(x)$. Take a PS sequence $x_n$, so $\nabla f(x_n)\to0$ and, since $\nabla b$ is compact, $\nabla b(x_n)\to z$. Therefore $Lx_n\to-z$. Since $L$ is invertible, $x_n\to-L^{-1}z$.
\end{proof}
\begin{definition}
	A functional $f\in C^2(\Hi)$ has the \emph{Morse-Smale property} on $I$ up to order $k$ if it is a Morse function on $I$ and the unstable and stable manifolds of every pair of critical points $x,y \in f^{-1}(I)$ such that $ind_{\Hi^{-}}(x)-ind_{\Hi^{-}}(y)\leq k$, meet transversally
\end{definition}
\begin{theorem}
	Assume that the functional $f\in\FF(L)$ satisfies PS and the Morse-Smale property up to order $k$ on the interval $I$. Let $x,y \in f^{-1}(I)$ be two critical points of $f$ such that $ind_{\Hi^{-}}(x)-ind_{\Hi^{-}}(y)\leq k$. Then $W^{u}(x)\cap W^{s}(y)$, if nonempty, is an embedded $C^1$-submanifold of $\Hi$ of dimension
	\[
		\dim\big(W^{u}(x)\cap W^{s}(y) \big) = ind_{\Hi^{-}}(x)-ind_{\Hi^{-}}(y).
	\]
	Moreover, we have the following:
	\begin{itemize}
		\item When $k\geq0$, $ind_{\Hi^{-}}(x)-ind_{\Hi^{-}}(y)\leq0$, and $x\neq y$, we have $W^{u}(x)\cap W^{s}(y)=\emptyset$.
		\item When $k\geq0$, and $ind_{\Hi^{-}}(x)-ind_{\Hi^{-}}(y)=1$, $W^{u}(x)\cap W^{s}(y)\cup\{x,y\}$ is compact.
	\end{itemize}
\end{theorem}
This theorem implies, in particular, that when $ind_{\Hi^{-}}(x)-ind_{\Hi^{-}}(y)=1$, there is a finite number of trajectories from $x$ to $y$.
The manifolds $W^{u}(x)\cap W^{s}(y)$ admit an orientation \cite[\S3.5]{AM} and thus, when  $ind_{\Hi^{-}}(x)-ind_{\Hi^{-}}(y)=1$, all the trajectories  from $x$ to $y$ come with a sign.

The idea of orientation is the following. Let $F_{p}(\Hi)$ denote the set of Fredholm pairs in $\Hi$. We have the non-trivial line bundle
\[
	\xymatrix{
		\Lambda^{\textrm{max}}(V\intt W)\otimes \Lambda^{\textrm{max}}\left(\Big(\frac{\Hi}{V+W}\Big)^{*}\right)  \ar@{^{(}->}[r] & \op{Det}\big(F_{p}(\Hi)\big)\ar[d]\\
		 & F_{p}(\Hi)
	}
\]
If $x$ is a critical point of the functional $f$, the pair $\big(T_{x}W^{u}(x),\Hi^{+}\big)$ is in $F_{p}(\Hi)$. We choose an orientation of the determinant line bundle over this pairs and we do the same at every (critical) point. This induces an orientation over $\big(T_{x}W^{s}(x),\Hi^{-}\big)$.\\
Thus, $\big(T_{p}W^{u}(x),\Hi^{+}\big)$ and $\big(T_{p}W^{s}(x),\Hi^{-}\big)$ are oriented for all $p\in W^{u}(x)$. This induces a canonical orientation of $\big(T_{p}W^{u}(x),T_{p}W^{s}(x)\big)$. If the functional is Morse-Smale, we are done.

When $ind_{\Hi^{-}}(x)-ind_{\Hi^{-}}(y)=1$, let $\#\mathcal{N}(x,y)$ denote the count, with signs, of trajectories from $x$ to $y$.

Given an interval $I$ of the extended real line and a functional $f$ satisfying the following conditions
\begin{enumerate}
	\item[(M.1)]\label{M1} $f\in\FF(L)$;
	\item[(M.2)] $f$ satisfies the PS condition on $I$;
	\item[(M.3)] $f$ is a Morse function on $I$;
	\item[(M.4)] $f$ has the Morse-Smale property on $I$ up to order 2;
	\item[(M.5)]\label{M5} for every $a\in I$ and every $k\in\Z$, the set $\op{crit}_k\big(f,I\cap(-\infty,a]\big)$ is finite;
\end{enumerate}
we can define a Morse homology of the pair $(f,I)$.
The Morse complex in degree $k$ is defined as
\[
	CM_k(f,I):=\oplus_{x\in\op{crit}_k(f,I)}\Q\langle x\rangle.
\]
and the boundary operator $\partial_{k}^{f,I}: CM_k(f,I)\to CM_{k-1}(f,I)$ is defined, for $x\in\op{crit}_k(f,I)$, as
\[
	\partial_{k}^{f,I}(x) = \sum_{y\in\op{crit}_{k-1}(f,I)}\#\mathcal{N}(x,y) y.
\]
\begin{theorem}
	Assuming the functional $f$ satisfies \hyperref[M1]{(M.1)}--\hyperref[M5]{(M.5)}, the boundary operator $\partial_{k}^{f,I}$ is an actual boundary homomorphism, i.e.
	\[
		\partial_{k}^{f,I}\circ\partial_{k}^{f,I}=0.
	\]
\end{theorem}
Therefore the pair $\big(CM(f,I),\partial_{k}^{f,I}\big)$ is a chain complex called the \emph{Morse complex} of $(f,I)$ and its homology is called the \emph{Morse homology} of $(f,I)$.

As for Morse theory in finite dimensions, changing the functional lead to a morphism between the Morse homologies.
Let $f_{0}$ and $f_{1}$ be two functionals and let $f_{s}$ be a homotopy interpolating between the two; $f_{s}=f_{0}$ for $s\leq\eps$ and $f_{s}=f_{1}$ for $s\geq1-\eps$. Let $\varphi:\R\to\R$ a smooth function with two critical points: a maximum at $0$, with $\varphi(0)=1$ and a minimum at $1$ with $\varphi(1)=0$. Let $\widetilde{f}:\R\times\Hi\to\R$ be the functional defined by $\widetilde{f}(s,x) = \varphi(s)+f_{s}(x)$. The critical points of $\widetilde{f}$ of index $k$ are
\[
	\op{crit}_{k}\widetilde{f} = \{0\}\times\op{crit}_{k-1}f_{0} \union \{1\}\times\op{crit}_{k}f_{1}.
\]
The associated differential $\partial_{\widetilde{f}}$ writes as
\[
	\partial_{\widetilde{f}} = \left(\begin{matrix} \partial_{f_{0}} & \phi\\ 0 & \partial_{f_{1}} \end{matrix}\right).
\]
This $\phi$ is precisely the continuation map (as in finite-dimensional Morse homology). It counts parametrized gradient trajectories between critical points of $f_0$ and $f_1$ of the same index.

\subsection{The case of the action functional for star-shaped domains}\label{HMAH}
We start with an admissible parametrized Hamiltonian $H$ as in Defintion \ref{def:aph}.
Let $\LL(\R^{2n}):=C^{\infty}(S^{1},\R^{2n})$ be the free loop space of $\R^{2n}$. The parametrized Hamiltonian action functional $\Ac_{H}$ on $\LL(\R^{2n})\times S^{2N+1}$ is defined as
\begin{equation}\label{AH2}
	\mc{A}_H(\gamma,z):=-\int_{\gamma}\widehat{\lambda}-\int_{S^1}H\bigl(\theta,\gamma(\theta),z\bigr)d\theta.
\end{equation}
To ensure, we have a Morse theory of this $\Ac_{H}$, we have to complete  $\LL(\R^{2n})$ to a Hilbert manifold; its structure will be induced by $H^{\frac{1}{2}}(S^1,\R^{2n})$.
Then we shall extend the functional $\Ac_{H}$ and check that it satisfies the 5 conditions \hyperref[M1]{(M.1)}--\hyperref[M5]{(M.5)} listed above.

\subsubsection{The Hilbert manifold}
Since $\LL(\R^{2n})\subset L^2(S^1,\R^{2n})$, every element $x\in\LL(\R^{2n})$ can be written as a Fourier series with coefficients in $\R^{2n}$.
\[
	x(t)=\sum_{k\in\Z}x_ke^{2\pi ikt}.
\]
Using this Fourier decomposition, $\LL(\R^{2n})$ can be completed in the Sobolev space: $H^{\frac{1}{2}}(S^1,\R^{2n})$ (which carries a Hilbert structure).
\[
	H^{\frac{1}{2}}(S^1,\R^{2n})\eqdef\left\{x\in L^2(S^1,\R^{2n})\,|\,\sum_{k\in\Z}|k|\|x_k\|^2<\infty \right\}.
\]
We have the orthogonal decomposition
\[
	H^{\frac{1}{2}}(S^1,\R^{2n}) = E^{+}\oplus E^{0}\oplus E^{-}
\]
 with respect to the inner product $\langle x,y\rangle := \langle x_0,y_0\rangle + 2\pi\sum_{0\neq k\in\Z}|k|\langle x_k,y_k\rangle$ and where
 \begin{align*}
 	E^{-} &= \{x\in H^{\frac{1}{2}}(S^1,\R^{2n})\,|\,x_k=0 \textrm{ for } k\geq0\}\\
	E^{0} &= \{x\in H^{\frac{1}{2}}(S^1,\R^{2n})\,|\,x_k=0 \textrm{ for } k\neq0\}\cong\R^{2n}\\
	E^{+} &= \{x\in H^{\frac{1}{2}}(S^1,\R^{2n})\,|\,x_k=0 \textrm{ for } k\leq0\}.
 \end{align*}
Let $P_{E^+}$, $P_{E^-}$ and $P_{E^0}$ denote the orthogonal projections on $E^{+}$, $E^{-}$ and $E^{0}$ respectively.

\subsubsection{The functional}
Recall the class $\FF(L)$ of functionals for which the Morse homology is defined. Let $\Hi$ be a real Hilbert space and let $L$ be a linear, invertible, self-adjoint operator on $\Hi$. We are looking at the functional $f:\Hi\to\R$
\[
	f(x) = \frac{1}{2} (Lx\,,\,x)+b(x)
\]
where $b$ is $C^2$ and $\nabla b:\Hi\to\Hi$ is a compact map.
In the case of a nice star-shaped domain in $\R^{2n}$, the Hilbert space is $\Hi=H^{\frac{1}{2}}\times S^{2N+1}$ and the functional is given by
\begin{equation}\label{AH}
	\Ac_{H}(x,z) = -\tfrac{1}{2}\int J\dot{x}\cdot x dt - \int_{0}^{1}{H\big(t,x(t),z\big)dt}.
\end{equation}
The fact that this functional coincides with the one from equation \eqref{AH2} is a direct computation.

Fixing $L$, we denote by $\Hi^{+}$ the maximal $L$-invariant subspace on which $L$ is positive and by $\Hi^{-}$ the maximal $L$-invariant subspace on which $L$ is negative. We have $\Hi=\Hi^{+}\oplus\Hi^{-}$.
Here $\Hi$ decomposes as $\Hi=(E^{+}\oplus E^{0}\oplus E^{-})\times S^{2N+1}$ where $E^{0}\cong\R^{2n}$ is the set of constant loops. We split $E^{0}$ arbitrarily in $E^{0}=E^{0}_{+}\oplus E^{0}_{-}$ where $E^{0}_{+}\cong\R^{n}\cong E^{0}_{-}$. In the previous notation, we take $\Hi^{+}=(E^{+}\oplus E^{0}_{+})\times S^{2N+1}$ and $\Hi^{-}=(E^{-}\oplus E^{0}_{-})\times S^{2N+1}$ by extending $L$ with the matrix $\left(\begin{matrix}\Id &0\\0&-\Id\end{matrix}\right)$.
\begin{remark}
	By taking the splitting of $E^{0}$ to be given by $E^{0}_{-}=\langle x_1,\ldots,x_n\rangle$, we have that the parametrized CZ-index is equal to the relative Morse index, see \cite{AbbondandoloBook}.
\end{remark}

The functional then writes as
\begin{align}
	\Ac_{H} (x,z)&= -\tfrac{1}{2}\int_{0}^{1}{J\dot{x}\cdot x dt} - \int_{0}^{1}{H\big(t,x(t),z\big)dt}\\
	&= \tfrac{1}{2}\big(\|P_{E^+}(x)\|^{2}_{H^{\frac{1}{2}}}-\|P_{E^-}(x)\|^{2}_{H^{\frac{1}{2}}}\big) - \int_{0}^{1}{H\big(t,x(t),z\big)dt}\\
	&= \tfrac{1}{2}(Lx\,,\,x)_{\tfrac{1}{2}} \underbrace{- \tfrac{1}{2}\|P_{E^{0}_{+}}x\|^{2} + \tfrac{1}{2}\|P_{E^{0}_{-}}x\|^{2} - \int_{0}^{1}{H\big(t,x(t),z\big)dt}}_{b}.\label{b}
\end{align}
\begin{proposition}\cite[Lemma 3.4]{HZ}
	The map $b:\Hi\to\R$ from equation \eqref{b} is differentiable. Its gradient $\nabla b:\Hi\to\Hi$ is continuous and compact.
\end{proposition}

Note that not all element of $H^{\frac{1}{2}}$ can be represented by a continuous function.
\begin{proposition}\cite[Proposition 3.4]{HZ}
	Let $s>\frac{1}{2}$. If $x\in H^{s}(S^1,\R^{2n})$, then $x\in C^0(S^1,\R^{2n})$. Moreover, there is a constant $c$, depending on $s$, such that
	\[
		\|x\|_{C^{0}}\leq c\|x\|_{H^{s}},\qquad\forall x\in H^{s}(S^{1},\R^{2n}).
	\]
\end{proposition}
Recall that from \cite[Proposition 3.3]{HZ}, for $t>s\geq0$, the inclusion map $I:H^{t}(S^{1},\R^{2n})\to H^{s}(S^{1},\R^{2n})$ is compact.

The following inclusion $j$, and its adjoint $j^{*}$, will play a key role in the following.
\begin{align*}
	j:H^{\frac{1}{2}}(S^{1},\R^{2n})&\to H^{0}(S^{1},\R^{2n})=L^{2}(S^{1},\R^{2n})\\
	j^{*}:L^{2}(S^{1},\R^{2n})&\to H^{\frac{1}{2}}(S^{1},\R^{2n})
\end{align*}
\begin{proposition}\cite[Proposition 3.5]{HZ}
	\[
		j^{*}\big(L^{2}(S^{1},\R^{2n})\big)\subset H^{1}(S^{1},\R^{2n})\quad\textrm{and}\quad \|j^{*}(y)\|_{H^{1}}\leq\|y\|_{L^{2}}.
	\]
\end{proposition}

\begin{proposition}\cite{HZ}
	The Hamiltonian action functional $\Ac_{H}:\Hi\to\R$ is a smooth functional. Its gradient (on $H^{\frac{1}{2}}$), with respect to the inner product on $H^{\frac{1}{2}}$ is given by
	\[
		\nabla_{\half}\Ac_{H_z}(x)= -P_{E^+}(x)+P_{E^{-}}(x)+j^{*}\nabla H_z\big(\cdot,x(\cdot)\big).
	\]
	Moreover, $\nabla_{\half}\Ac_{H_z}$ is Lipschitz continuous on $\Hi$ with uniform Lipschitz constant. Its Jacobian is given by
	\[
		\nabla^{2}_{\half}\Ac_{H_z}(x)= -P_{E^+}+P_{E^{-}}+j^{*}\nabla^{2} H_z\big(\cdot,x(\cdot)\big).
	\]
\end{proposition}
We denote by $X\eqdef-\nabla_{\half}\Ac_{H}$ the gradient vector field of the Hamiltonian action functional. Assume that $(x,z)\in\Hi$ is a critical point of the action functional; i.e.\ $\nabla_{\half}\Ac_{H_z}(x)=0$. Then $x\in C^{\infty}(S^{1},\R^{2n})$. Moreover it solves the Hamiltonian equation
\[
	\dot{x}(t)=J\nabla H_z\big(x(t)\big).
\]
\begin{lemma}\cite[Lemma 3.7]{HZ}
	The flow of $\dot{x}=X(x)$ is globally defined, maps bounded sets to bounded sets and admits the representation
	\[
		(x,z)\cdot t=e^tx^{-}+x^{0}+e^{-t}x^{+}+K(t,x,z)
	\]
	where $K:\R\times\Hi\to\Hi$ is continuous and maps bounded sets in precompact sets.
\end{lemma}

\subsubsection{The conditions (M.1)--(M.5) are satisfied by $\Ac_{H}$}
Condition (M.1) is satisfied and the decomposition $\Hi=\Hi^+\oplus\Hi^-$ is as in \hyperref[sec:generalmorse]{\S\ref*{sec:generalmorse}}. To ensure the Morse property, we have to pick a generic Hamiltonian
\begin{proposition}[\cite{AbMa2}]
	There is a residual set (in the sense of Baire) $\Hi_{reg}\subset C^\infty(S^1\times\R^{2n}\times S^{2N+1},\R)$ of Hamiltonians such that the negative $H^{\frac{1}{2}}$-gradient $X$ of $\Ac_H$ is a Morse vector field for evey $H\in\Hi_{reg}$. In particular, the set of critical points of $\Ac_H$ is a finite set.
\end{proposition}

To ensure transversality (condition (M.4)), we need to perturb the vector field $X=-\nabla_{\half}\Ac_{H}$ by adding a small compactly supported vector field $\overline{X}$. We do it this way rather than following \cite{AM} in preparation for transversality for hybrid-type curves \hyperref[sec:isomorphism]{\S\ref*{sec:isomorphism}}.
Let $\mathcal{K}(\Hi)\subset C^3_ b(\Hi)$ be the closed subspace of all $C^3$-vector fields which are compact and bounded on $\Hi$. We choose a $C^1$-function $g:\Hi\to\R^+$ satisfying
\begin{enumerate}[leftmargin=*]
	\item $g(p)>0$ everywhere else; i.e. for all $p\in\Hi\setminus\op{Crit}\Ac_H$,
	\item $g(p)\leq\frac{1}{2}\|\nabla_{\frac{1}{2}}\Ac_H(p)\|_{H^{\frac{1}{2}}}$ for all $p\in\Hi$.
\end{enumerate}
In particular, we have $g(x)=0$ for all $x\in\op{Crit}\Ac_H$.
We consider the subset of vector fields
\[
	\mathcal{K}_g\eqdef\left\{\overline{X}\in\mathcal{K}(\Hi)\,|\,\exists c>0\textrm{ such that }\|\overline{X}_p\|_{H^{\frac{1}{2}}}\leq cg(p)\quad\forall p\in\Hi\right\}.
\]
This set is a Banach space when equipped with the following norm:
\[
	\|\overline{X}\|_{\mathcal{K}_g}\eqdef\sup_{p\in\Hi\setminus\op{Crit}(\Ac_H)}\frac{\|\overline{X}_p\|_{H^{\frac{1}{2}}}}{g(p)}+\|\nabla\overline{X}\|_{C^2}.
\]
We denote the open unit ball in $\mathcal{K}_g$, with respect to the above norm, by $\mathcal{K}_{g,1}$. It is a Banach manifold with trivial tangent bundle.
\begin{lemma}
	Let $\overline{X}\in \mathcal{K}_{g,1}$ and let $\widetilde{X}\eqdef-\nabla_{\frac{1}{2}}\Ac_H+\overline{X}$. Then
	\begin{enumerate}[leftmargin=*]
		\item The singular points of $\widetilde{X}$ are the critical points of the action functional
			\[
				\op{sing}(\widetilde{X})=\op{Crit}(\Ac_H).
			\]
		\item For all $x\in\op{Crit}(\Ac_H)$, we have
			\[
				D\widetilde{X}(x)=-D^2\Ac_H(x).
			\]
		\item The action functional is a Lyapunov function for $\widetilde{X}$; i.e.
			\[
				D\Ac_H(p)\big(\widetilde{X}(p)\big)<0\qquad\textrm{for all }p\in\Hi\setminus\op{Crit}(\Ac_H).
			\]
	\end{enumerate}
\end{lemma}
\begin{theorem}
	There is a residual subset $\mathcal{K}_{reg}\subset\mathcal{K}_{g,1}$ of compact vector fields $\overline{X}$ such that the perturbed vector field $\widetilde{X}\eqdef-\nabla_{\frac{1}{2}}\Ac_H+\overline{X}$ fulfills the Morse-Smale condition up to order 2.
\end{theorem}
We are therefore in a situation where we can define the Morse homology as in \hyperref[sec:generalmorse]{\S \ref*{sec:generalmorse}}.
We shall use the notation $CM_{*}^{S^1}(H)$ for the Morse complex associated with the parametrized Hamiltonian action functional $\Ac_H$.


\section{The isomorphism}\label{sec:isomorphism}
The main result of this section is that the Floer complex and the Morse complex associated to the parametrized action functional are isomorphic.
\begin{theorem}\label{thm:maintechnical}
	Let $H:S^1\times\R^{2n}\times S^{2N-1} \to \R$ be an admissible parametrized smooth Hamiltonian function as in Definition \ref{def:aph}. Then there exists a chain complex isomorphism
	\[
		\Phi: \left(CM_{*}^{S^1}(H),\partial^M\right)\to \left(CF_*^{S^1}(H),\partial\right).
	\]
	Moreover, $\Phi$ commutes with continuation maps.
\end{theorem}

We shall define this chain map $\Phi: CM^{S^1}(H)\to CF^{S^1}(H)$ by counting parametrized hybrid curves in a similar manner as \cite{AK}.

Let $Z=[0,\infty)\times S^1$ and let $(x,p),(y,q)\in\op{Crit}(\Ac_H)$. We also let $H$ and $X$ be generic as explained in \hyperref[sec:symplectichomology]{Section \ref*{sec:symplectichomology}}. We define

	\begin{align*}
	\Mod_{hyb}((x,p),&(y,q),X)=\\
	&\big\{(u,\eta)\in  H^1_{loc}(Z,\R^{2n})\times H^1_{loc}([0,\infty),S^{2N+1})\mid\\ &\phantom{abcdef}\bar{\partial}_{J_0,H_\eta}(u)=0,\dot{\eta}+\nabla f_N(\eta)=0,\\
	&\phantom{abcdef}(u(0,\cdot),\eta(0))\in W^u_X(x,p),\\
	&\phantom{abcdef}\lim _{s\to\infty} (u(s,\cdot),\eta(s))=(y,q)\big\}\\
	\end{align*}

As before, generically, $\Mod_{hyb}((x,p),(y,q))$ is a smooth manifold of dimension $\op{ind}_{\mathcal{H}^-}(x,p)-(\op{ind}(f_N,q)-\CZ(y))$; moreover if $\op{ind}_{\mathcal{H}^-}(x,p)=\op{ind}(f_N,q)-\CZ(y)$ then the manifold is compact and thus consists of finitely many points.
So we let
\[\Phi(x,p)=\sum_{\op{ind}_{\mathcal{H^-}}(x,p)=\op{ind}(f_N,q)-\CZ(y)}\#\Mod_{hyb}((x,p),(y,q))\cdot (y,q).\]


The rest of this section is devoted to the proof of \hyperref[thm:maintechnical]{Theorem \ref*{thm:maintechnical}}.
We start by showing that the map $\Phi$ is well-defined i.e. that whenever $\op{ind}_{\mathcal{H^-}}(x,p)=\op{ind}(f_N,q)-\CZ(y)$, $\Mod_{hyb}((x,p),(y,q))$ is a smooth compact manifold of dimension 0 (transversality and compactness).
We shall then prove that $\Phi$ is an isomorphism.

\subsection{Transversality}\label{sec:trans}
In this section we show that we can perturb $H$ and the vector field $-\nabla \A_H$ to achieve transversality. The main result is that although the boundary value problem is not Lagrangian, the linearized operator is still Fredholm.

We will now apply an idea from \cite{Hecht} to achieve transversality. We start by recalling the following classical result.
\begin{prop}[\cite{AbMa2}]
There exists a residual set $\mathcal{H}_{reg}\subset \mathcal{H}$ such that $\A_H$ is a Morse function, i.e., every $(x,p)\in\op{Crit}(\Ac_H)$ is nondegenerate.
\end{prop}

The goal of this section is to prove the following theorem.
\begin{theorem}\label{thm:transv}
Let $H\in\mathcal{H}_{reg}$ and $(x,p),(y,q)\in\op{Crit}(\Ac_H)$. Then for a generic $C^3$-small compact perturbation $X$ of $-\nabla \A_H$, the moduli space $\Mod_{hyb}((x,p),(y,q),X)$ is a $C^3$ manifold of dimension $\op{ind}_{\mathcal{H}^-}(x,p)-\op{ind}(f_N,q)+\CZ(y)$ and $X$ satisfies the Morse-Smale condition up to order 2.
\end{theorem}

In order to prepare for the proof of Theorem \ref{thm:transv}, we now clarify the definition of $\Mod_{hyb}((x,p),(y,q),X)$ and write it as an affine translation of the zero set of a function between Hilbert spaces. We first observe that the evaluation $u\mapsto u(s,\cdot)$ for smooth maps extends to a map $H^1(Z,\R^{2n})\to H^{1/2}(S^1,\R^{2n})$. So the limit of $u(s,\cdot)$ as $s\to\infty$ is taken in $H^{1/2}(S^1,\R^{2n})$ with its standard topology. For $(x,p)\in\op{Crit}(\Ac_H)$, let $c_{x,p}=(c_p,c_x)$ where $c_{x}\in C^\infty(Z,\R^{2n})$ such that $c_{x}(s,\cdot)= 0$ for $s\le 1$ and $c_{x}(s,\cdot)=x$ for $s\ge 2$ and $c_{p}\in C^\infty(\R_{\geq0},S^{2N+1})$ such that $c_{p}(s)= 0$ for $s\le 1$ and $c_{p}(s)=$ for $s\ge 2$. We now let
\[
	H^1_{W^u_X(x,p)}=\{(\eta,u)\in H^1(\R_{\geq0},S^{2N+1})\times H^1(Z,\R^{2n})\mid (\eta(0),u(0,\cdot)\in W^u_X(x,p)\}.
\]

\begin{prop}
Let $H\in \mathcal{H}_{reg}$, $x,y\in\Per(H)$ and $X$ a nondegenerate gradient-like vector field on $E$ of class $C^1$ with globally defined flow. Then
\begin{enumerate}
\item $H^1_{W^u_X(x,p)}$ is a $C^1$ Hilbert submanifold of $H^1(\R_{\geq0},S^{2N+1})\times H^1(Z,\R^{2n})$
\item The function \[\Theta_{(x,p),c_{y,q}}:H^1_{W^u_X(x,p)}\to L^2(\R_{\geq0},S^{2N+1})\times L^2(Z,\R^{2n}),\quad (\zeta,w)\mapsto(\nabla f_N(\zeta+c_q),\bar{\partial}_{J_0,H_\zeta}(w+c_y))\] is well-defined of class $C^1$ and \[\Mod_{hyb}((x,p),(y,q),X)=c_{y,q}+\Theta_{(x,p),c_{y,q}}^{-1}(0).\]
\end{enumerate}
\end{prop}

\begin{proof}
For every $\chi\in W^u_X(x,p)$, its tangent space $T_\chi W^u_X(x,p)$ is a compact perturbation of $(E^-\oplus E^0_-)\times\R^{2N+1}$, which is complemented by $E^+\times\R^{2N+1}$. So (a) follows from the implicit function theorem.

To see that $\Theta_{(x,p),c_{y,q}}$ is well-defined, let $(\zeta,w)\in H^1_{W^u_X(x,p)}$. Then
\begin{equation}\label{eq:theta1}
\Vert\Theta_{(x,p),c_{y,q}}(\zeta,w)\Vert^2_{L^2}=\Vert \partial_s(w+c_y)\Vert^2_{L^2}+\Vert\partial_t(w+c_y)-X_H(w+c_y)\Vert^2_{L^2}+\|\nabla f_N(\zeta+c_q)\|^2_{L^2}
\end{equation}
We observe that
\[\begin{aligned}\Vert \partial_s(w+c_y)\Vert_{L^2(Z)} &\le C_1+\Vert \partial_s w\Vert_{L^2([2,\infty)\times S^1)}\le C_1+\Vert w\Vert_{H^1(Z)},\\\Vert\partial_t(w+c_y)-X_H(w+c_y)\Vert_{L^2(Z)}&\le C_2+ \Vert\partial_t w+\partial_t c_y-X_H(w+c_y)\Vert_{L^2([2,\infty)\times S^1)}\\
&\le C_2+\Vert w\Vert_{H^1(Z)}+\Vert D^2H\Vert_{C^0(S^1\times\R^{2n})}\Vert w\Vert_{L^2(Z)}\end{aligned}
\]

The $C^1$-smoothness of $\Theta_{x,c_y}$ follows from the smoothness of $\bar{\partial}_{J_0,H}$. The last statement follows from the exponential decay of solutions to the $\bar{\partial}$-equation.
\end{proof}

We can now prove Theorem \ref{thm:transv}.
\begin{proof}
We first claim that the linearized operator 
\[D\Theta_{(x,p),c_{y,q}}(\zeta,w):H^1_{W^u_X(x,p)}\to L^2(\R_{\geq0},S^{2N+1})\times L^2(Z,\R^{2n})\]
is a Fredholm operator whose Fredholm index is 
\[\text{ind} D\Theta_{(x,p),c_{y,q}}(\zeta,w)=\op{ind}_{\mathcal{H}^-}(x,p)-\op{ind}(f_N,q)+\CZ(y).
\]
This claim is a slight modification of \cite[Theorem 4.4]{Hecht}, where we add the standard computation of the Fredholm index Hessian of $f_N$.

The result now follows from the implicit function theorem.
\end{proof}

\subsection{Compactness}
In this section, we prove that every sequence of elements in $\Mod_{hyb}$ admits a partial sequence which converges to a ``broken'' hybrid trajectory, more precisely:
\begin{theorem}\label{thm:compactness}
	Let $H$ be an admissible Hamiltonian and let $(x,p),(y,q)\in\op{Crit}(\Ac_H)$. Let $(u_n,\eta_n)_{n\in\N}\in\Mod_{hyb}((x,p),(y,q),X)$ be a sequence of hybrid trajectories. Let $\varphi_X(s,u_n,\eta_n)$, $s<0$, be the trajectories through $(u_n(0,\cdot),\eta_n(0))\in E\times S^{2N+1}$. Then there exists 
	\begin{itemize}
		\item$k+l$ elements of $\op{Crit}(\Ac_H)$, $(x,p)=(x_0,p_0),(x_1,p_1),\ldots,(x_k,p_k)$ and $(y_0,q_0),\ldots,(y_l,q_l)=(y,q)$, with
			\[
				\A_H(x_0,p_0)>\A_H(x_1,p_1)>\cdots>\A_H(x_k,p_k)>\A_H(y_0,q_0)>\cdots>\A_H(y_l,q_l),
			\]
		\item connecting trajectories
			\[
				V_1\subset W^u(x_0,p_0)\cap W^s(x_1,p_1),\ldots,V_p\subset W^u(x_{k-1},p_{k-1})\cap W^s(x_k,p_k),
			\]
		\item curves
			\[
				U_2\in\Mod((y_0,q_0),(y_1,q_1),J),\ldots,U_l\in\Mod((y_{l-1},q_{l-1}),(y_l,q_l),J),	
			\]
		\item and a hybrid trajectory
			\[
				U_1\in\Mod_{hyb}((x_k,p_k),(y_0,q_0),X)
			\]
	\end{itemize}
	such that there exists a subsequence $(u_{n_m},\eta_{n_m})$ of $(u_n,\eta_n)$ which converges to $(V_1,\ldots,V_p,U_1,U_2,\ldots,U_q)$ in the following sense:
	There exists reparametrization times $\tau_m^1,\ldots,\tau_m^k\subset(-\infty\,,\,0]$ and $\sigma_m^1,\sigma_m^l\subset[0\,,\,\infty)$ with $m\in\N$ such that
	\[
		\varphi_X(\tau_m^1,u_{n_m},\eta_{n_m})\to V_1,\ldots,\varphi_X(\tau_m^k, u_{n_m},\eta_{n_m})\to V_k\quad\textrm{in the Hausdorff distance}
	\]
	and
	\[
		u_{n_m},\eta_{n_m}\to U_1, u_{n_m},\eta_{n_m}(\cdot+\sigma_m^1,\cdot)\to U_2,\ldots,u_{n_m},\eta_{n_m},(\cdot+\sigma_m^l,\cdot)\to U_l\quad\textrm{in }H^1_{loc}.
	\]
\end{theorem}
\begin{proof}
	By \cite{Schwarz:1995yq}, we need to show that $\Mod_{hyb}((x,p),(y,q),X)$ is an $H^1_{loc}$-precompact set. Then \cite{AM} gives the desired result.
	
	We prove the $H^1_{loc}$-precompactness via the following two inequalities:
	\begin{equation}\label{eq:step1compactness}
		\|w\|_{L^2(Z_T)}\leq C
	\end{equation}
	where $Z_T:=[0,T]\times S^1\subset Z$ and $C$ is a constant depending on $x,y,T,H,c_y$ and $n$.
	\begin{equation}\label{eq:step2compactness}
		\|\delta w\|_{H^1(Z_{T'},\R^{2n})}\leq C\Big(\|\delta w\|_{L^2(Z_{T},\R^{2n})}+\|P^-\delta w(0,\cdot)\|_{H^{\frac{1}{2}}(S^1)}\Big)
	\end{equation}
	for $T>T'>0$, $\delta w:=w_1-w_2$ with $w_1,w_2\in\Theta_{x,c_y}^{-1}(0)$ and a constant $C$ depending on $T,T'$ and $H$.
	
	Let us first conclude the proof assuming \eqref{eq:step1compactness} and \eqref{eq:step2compactness} hold.
	The first inequality, \eqref{eq:step1compactness}, imply that the restriction to $Z_T$ of any sequence $(w_n)_{n\in\N}$ in $\Theta_{x,c_y}^{-1}(0)$ is uniformly bounded in $H^1(Z_{T},\R^{2n})$.
	Using that the embedding
	\[
		H^1_{W^u(x)}(Z_T,\R^{2n})\hookrightarrow L^2(Z_T,\R^{2n})
	\]
	is compact \cite{AK} and that the set $P^-\big(W^u(x)\cap\big\{\A_H\geq\A_H(y)\big\}\Big)$ is precompact \cite{Hecht}, inequality \eqref{eq:step2compactness} implies that the restriction to $Z_{T'}$ of the sequence $(w_n)_{n\in\N}$ admits an $H^1(Z_{T'},\R^{2n})$ convergent subsequence which concludes the proof.
\end{proof}
It remains to prove the two estimates, \eqref{eq:step1compactness} and \eqref{eq:step2compactness}.
The proof of both estimates will be a sequence of lemmas. We follow the ideas from \cite{Hecht} adapting the results to our context.
The next very general statement is a maximum principle for Floer trajectories.
\begin{lemma}[Abouzaid, \cite{Rit}]\label{thm:abouzaid}
	Let $(W',\w'=d\lambda')$ be an exact  symplectic manifold with contact type boundary $\partial W'$, such that the Liouville vector field points inwards.
	Let $\rho$ be the coordinate near $\partial W'$ defined by the flow of the Liouville vector field starting from the boundary and let $r:=e^\rho$; near the boundary the symplectic form writes   $\w'=d(r\alpha)$
	with $\alpha$ the contact form on $\partial W'$ given by the restriction of $\lambda'$.
	Let $J$ be a compatible almost complex structure such that $J^*\lambda'= dr$ on the boundary.\\
	a) Let  $H:W'\rightarrow \R$ be non negative, and such  that  $H=h(r)$ where $h$ is a convex increasing function near the boundary.
	Let $S$ be a compact Riemann surface with boundary and let $\beta$ be a 1-form such that  $d\beta \geq 0$.
	Then any solution $u:S\rightarrow W'$ of $(du - X_{H}\otimes\beta)^{0,1}=0$ with $u(\partial S) \subset \partial W'$
	is entirely contained in $\partial W'$.\\
	b) Let  $H:\R\times S^1\times W'\rightarrow \R$ be an increasing homotopy, such  that  $H(s,\theta,p,\rho)$ $=H_s^\theta (p,\rho)= h_s(r) $ where $h_s$ are convex increasing functions near the boundary. Let $S$ be a compact Riemann surface with boundary embedded in the cylinder {($\R\times S^1$ with the standard structure)}.
Then any solution $u:S\rightarrow W'$ of $(du - X_{H_s}\otimes d\theta)^{0,1}=0$ with $u(\partial S) \subset \partial W'$
is entirely contained in $\partial W'$.
\end{lemma}
The consequence of lemma \ref{thm:abouzaid} and the special form of our Hamiltonians, is the existence of a $R>0$ such that the image of all curves is in the ball $B_R$ of radius $R$.
\begin{lemma}\label{lem:boundxh}
	Let $Z_T:=[0,T]\times S^1\subset Z$ and let $u=w+c_{y}$ be a curve, we have
	\[
		\|X_H(u)\|_{L^2(Z_T)}\leq\sqrt{T}\|H\|_{C^1(S^1\times B_R)}.
	\]
\end{lemma}
\begin{proof}
	This is an immediate consequence from the form of $H$. Recall that $H(t,x)=a\Vert x\Vert^2$ if $\Vert x\Vert\ge R$, then
	\begin{align*}
		\|X_H(u)\|_{L^2(Z_T)} &= \left(\int_{S^1}\int_0^T|X_H(u)|^2\right)^{\tfrac{1}{2}}\leq\left(\int_{S^1}\int_0^T|2ar(u)|^2\right)^{\tfrac{1}{2}}\\
		&\leq\left(\int_{S^1}\int_0^T|2aR|^2\right)^{\tfrac{1}{2}}=\sqrt{T}2aR\\
		&=\sqrt{T}\|H'\|_{C^0(S^1\times B_R)}\leq\sqrt{T}\|H\|_{C^1(S^1\times B_R)}.
	\end{align*}
\end{proof}
\begin{lemma}\label{lem:4.12hecht}{\emph{(\cite[Lemma 4.12]{Hecht})}}
	Let the $x$ and $y\in\Per(H)$. Let $\Theta_{x,c_y}$ be the operator
	\[
		\Theta_{x,c_y}: H^1_{W^u(x)}(Z,\R^{2n})\to L^2(Z,\R^{2n}) :\quad w\mapsto\overline{\partial}_{J_0,H}(w+c_y).
	\]
	If we denote by $P^0(w)(s,\cdot)\in H^1\big([0\,,\,\infty),\R^{2n}\big)$ the constant part of $w\in\Theta_{x,c_y}^{-1}(0)$, then $P^0(w)\in C^0\big([0\,,\,\infty),\R^{2n}\big)$ and there exists a constant $C$, depending solely on $x$ and $y$, such that
	\[
		\|P^0(w)\|_{C^0([0\,,\,\infty),\R^{2n})}\leq C\quad\forall w\in\Theta_{x,c_y}^{-1}(0).
	\]
\end{lemma}
The proofs of \eqref{eq:step1compactness} and \eqref{eq:step2compactness} follows \cite{Hecht} closely but we recall them for completeness.
\begin{proof}[Proof of \eqref{eq:step1compactness}]
	Let $w\in\Theta_{x,c_y}^{-1}(0)$. We shall prove that for any bounded domain $Z_T:=[0\,,\,T]\times S^1\subset Z$, $w|_{Z_T}$ is uniformly bounded.
	
	Let $u=w+c_y$. The energy of $u$ is given by
	\[
		E(u)=\A_H\big(u(o,\cdot)\big)-\A_H(y)\leq\A_H(x)-\A_H(y).
	\]
	We thus compute $\|\partial_sw\|_{L^2(Z_T)}$ and $\|\partial_tw\|_{L^2(Z_T)}$.
    \[
        \|\partial_sw\|_{L^2(Z_T)}\leq \sqrt{\A_H(x)-\A_H(y)}+\|\partial_sc_y\|_{L^2(Z_T)}
    \]
	We denote the right hand side by $c_0(x,y,c_y)$.
    \[
        \|\partial_tw\|_{L^2(Z_T)}\leq\sqrt{\A_H(x)-\A_H(y)}+\sqrt{T}\|H\|_{C^1(S^1\times B_R)}+\|\partial_tc_y\|_{L^2(Z_T)}
    \]
	We denote this right hand side by $c_1(x,y,c_y,T,H)$.
	
	Looking at $P^0(w)$, we note that
    \[
        \|{P^0}^\perp(w)\|^2_{L^2(Z_T)}\leq c_1^2.
    \]
	This actually shows that the $t$-derivative bounds the nonconstant pat of $w$. By \hyperref[lem:4.12hecht]{Lemma \ref*{lem:4.12hecht}}, we have the inequality $|w_0(0)|\leq C$ with $C>0$, depending on $x$ and $y$.
	Therefore, letting $w_0(s):=w_0(0)+\int_0^s\partial_\sigma w_0(\sigma)d\sigma$ and using H\"older's inequality, we have
    \[
        \|w_0\|^2_{L^2(Z_T)}\leq2\int_0^T\bigg(|w(0)|^2+\Big(\int_0^s|\partial_\sigma w_0(\sigma)|d\sigma\Big)^2\bigg)ds
    \]
\end{proof}
\begin{proof}[Proof of \eqref{eq:step2compactness}]
	By \eqref{eq:step1compactness}, we only need to bound  $\|\nabla\delta w\|_{L^2(Z_{T'})}$. To do so, we are going to use a cut-off function $\beta(s)$ and bound it on $Z_T$.
	Let $\beta:\R\to\R$ be a smooth function such that
	\[
		\beta{s}=\begin{cases}
			1&\textrm{if }s\leq T'\\
			0&\textrm{if} s\geq T'+\frac{T-T'}{2}
		\end{cases}.
	\]
	We have
	\begin{align*}
		\|\nabla\delta w\|_{L^2(Z_{T'})}^2 
		&=\|\overline{\partial}_{J_0}(\beta\delta w)\|_{L^2(Z_T)}^2+\|P^-\delta w(0,.)\|_{H^{\frac{1}{2}}(S^1)}^2 - \|P^+\delta w(0,.)\|_{H^{\frac{1}{2}}(S^1)}^2\\
		&\qquad \textrm{ by \cite[Lemma 3.10]{Hecht}}\\
		&\leq 2(1+\|\beta'\|_{C^0(\R)}^2)\left(\|\overline{\partial}_{J_0}(\beta\delta w)\|_{L^2(Z_T)}^2+\|P^-\delta w(0,.)\|_{H^{\frac{1}{2}}(S^1)}^2 + \|\delta w\|_{L^2(Z_T)}\right)
	\end{align*}
	Using now the mean value theorem, we have
	\begin{align*}
		\|\overline{\partial}_{J_0}(\beta\delta w)\|_{L^2(Z_T)}^2 
		&\leq \|H\|_{C^2(S^1\times\R^{2n}}\|\delta w\|_{L^2(Z_T)}.
	\end{align*}
	We can thus find a constant $\widetilde{C}$, depending on $T,T'$ and $H$ such that
	\[
		\|\nabla(\beta\delta w)\|_{L^2(Z_{T})}^2\leq\widetilde{C}\left(\|\delta w\|^2_{L^2(Z_T)}+\|P^-\delta w(0,.)\|^2_{H^{\frac{1}{2}}(S^1)}\right)
	\]
	and therefore the conclusion.
\end{proof}

\subsection{Bijection}
First, $\Phi$ is a homomorphism. This follows from the fact that the closure of the moduli space of hybrid curves between two critical points of index difference 1 is a 1-dimensional manifold whose boundary consists of all the broken trajectories. See Theorem \ref{thm:compactness}.

Secondly, note that if $(x,p)$ is a critical point of $\Ac_H$, then $\mathcal{M}_{hyb}((x,p),(x,p))$ only consists of the constant solution. This is an immediate consequence of the Crossing Energy Theorem. Secondly, by construction, the morphism $\Phi$ decreases $\Ac_H$.
To prove that $\Phi$ is an isomorphism, let $(x_1,p_1),\ldots,(x_m,p_m)$ be the critical points of $\Ac_H$ of index $k$ arranged in increasing action. Then, we have
\[
	\Phi_k=\left(\begin{array}{cccc}
		1&*&\cdots&*\\
		0&1&\cdots&*\\
		\vdots&\ddots&\ddots&\vdots\\
		0&\cdots&0&1
	\end{array}\right)
\]

Finaly, $\Phi$ commutes with the $U$-map follows mutatis mutandis from \cite{GuH}.

\bibliographystyle{alpha}
\bibliography{bibliography.bib}

\end{document}